\documentclass[twoside,11pt]{article}

\usepackage{blindtext}

\usepackage{clean-version}
\usepackage[utf8]{inputenc} 
\usepackage[T1]{fontenc}    
\usepackage{hyperref}       
\usepackage{url}            
\usepackage{amsmath, amssymb,amsfonts}
\usepackage{graphicx}
\usepackage{booktabs} 
\usepackage{nicefrac}
\usepackage{comment}
\usepackage{mathtools}
\usepackage{xcolor}

\newcommand{\todo}[1]{#1}
\newcommand{\mv}[1]{{\boldsymbol{\mathrm{#1}}}}

\DeclareMathOperator{\sgn}{sgn}
\newcommand{\Cov}{\operatorname{Cov}}
\newcommand{\Var}{\operatorname{Var}}

\usepackage{lastpage}
\pubheading{xx}{2024}{1-\pageref{LastPage}}{xx; Revised xx}{xx}{xx}{Bolin, Mehandiratta and Simas}

\ShortHeadings{Linear cost approximation of Mat\'ern processes}{Bolin, Mehandiratta and Simas}
\firstpageno{1}

\begin{document}

\title{Linear cost and exponentially convergent approximation of Gaussian Mat\'ern processes \todo{on intervals}}

\author{\name David Bolin \email david.bolin@kaust.edu.sa \\
       \addr CEMSE Division, statistics program\\
       King Abdullah University of Science and Technology (KAUST),\\
       Thuwal 23955-6900, Kingdom of Saudi Arabia
       \AND
       \name Vaibhav Mehandiratta \email vaibhav.mehandiratta@kaust.edu.sa \\
       \addr CEMSE Division, statistics program\\
       King Abdullah University of Science and Technology (KAUST),\\
       Thuwal 23955-6900, Kingdom of Saudi Arabia
       \AND 
       \name Alexandre B. Simas \email alexandre.simas@kaust.edu.sa\\
       \addr CEMSE Division, statistics program\\
       King Abdullah University of Science and Technology (KAUST),\\
       Thuwal 23955-6900, Kingdom of Saudi Arabia}

\editor{}

\maketitle

\begin{abstract}
The computational cost for inference and prediction of statistical models based on Gaussian processes with Matérn covariance functions scales cubically with the number of observations, limiting their applicability to large data sets. The cost can be reduced in certain special cases, but there are no generally applicable exact methods with linear cost. Several approximate methods have been introduced to reduce the cost, but most lack theoretical guarantees for accuracy.
We consider Gaussian processes on bounded intervals with Matérn covariance functions and, for the first time, develop a generally applicable method with linear cost and a covariance error that decreases exponentially fast in the order $m$ of the proposed approximation. The method is based on an optimal rational approximation of the spectral density and results in an approximation that can be represented as a sum of $m$ independent Gaussian Markov processes, facilitating usage in general software for statistical inference.
Besides theoretical justifications, we demonstrate accuracy empirically through carefully designed simulation studies, which show that the method outperforms state-of-the-art alternatives in accuracy for fixed computational cost in tasks like Gaussian process regression.
\end{abstract}


\section{Introduction}
Gaussian stochastic processes with Mat\'ern covariance functions \citep{matern60} are important models in statistics and machine learning \citep{porcu2023mat}, and in particular in areas such as 
spatial statistics \citep{stein1999interpolation, lindgren2022spde}, computer experiments \citep{santner2003design, gramacy2020surrogates} and Bayesian optimization \citep{srinivas2009gaussian}. \todo{Although the Mat\'ern covariance often is used for spatial data, it is also commonly used for temporal data, in particular in areas such as functional data analysis \citep{Jingjing2016}, longitudinal data analysis \citep{asar2020jrssc}, and growth rate modeling \citep{Swain2016}.}
A Gaussian process $u$ on $\mathbb{R}$ has a Mat\'ern covariance function if 
\begin{equation}\label{eq:matern_cov}
 \Cov(u(s),u(t)) = \varrho(|s-t|;\nu,\kappa,\sigma^2), \quad \varrho(h;\nu,\kappa,\sigma^2)=\frac{\sigma^2}{2^{\nu-1}\Gamma{(\nu)}}(\kappa h)^{\nu} K_\nu(\kappa h),
\end{equation}
where $K_\nu(\cdot)$ is a modified Bessel function of the second kind of order $\nu$, and $\Gamma(\cdot)$ denotes the gamma function. The three parameters $\kappa, \sigma^2, \nu>0$, determine the practical correlation range, variance and smoothness of the process, respectively. \todo{We introduce the practical correlation range $\rho = \sqrt{8\nu}/\kappa$ as the distance at which the correlation is approximately 0.1, and a new smoothness parameter $\alpha = \nu + \nicefrac12$, which will be used throughout this text.}

 However, the approach of defining Gaussian processes via the covariance function incurs a significant computational cost for inference and prediction, mainly due to the requirement of factorising dense covariance matrices.
To address this so called ``Big n'' problem, referring to the $\mathcal{O}(n^3)$ computational cost required for problems with $n$ observations, several methods have been developed in recent years to reduce the cost. 
In this work, we \todo{consider} the case where the data is observed on an interval $I \subset \mathbb{R}$.  The processes have Markov properties when $\alpha \in \mathbb{N}$, and this has been used to derive exact methods with a $\mathcal{O}(n)$ cost for these particular cases, see for example the Kernel packet method \citep{chen2022kernel} or the state-space methods of \cite{hartikainen2010kalman,sarkka2012infinite,sarkka2013spatiotemporal}. If the data is evenly spaced in $\mathbb{R}$, Toeplitz methods \citep{wood1994simulation} can also be used to reduce the computational cost to \todo{$\mathcal{O}(n(\log n)^2)$}  \citep{Ling, supergauss}. Outside these particular cases, there are currently no exact methods that can reduce the computational cost, and one instead needs to rely on approximations. One popular method is the random Fourier features approach of \cite{rahimi2007random} which gives a cost $\mathcal{O}(m^2+mn)$ when using $m$ features, and an accuracy of $\mathcal{O}(m^{-\nicefrac{1}{2}})$. This is one of the few methods that have theoretical guarantees for how quickly the approximation error decreases as the order $m$ increases. Unfortunately, the theoretical rate is low and as we will later see, the method provides poor approximations for Mat\'ern processes on $\mathbb{R}$. 

Another widely used method is the SPDE approach of \cite{lindgren11}, where the Mat\'ern process is represented as a solution to a stochastic differential equation (SDE) which is approximated via a finite element method (FEM) approximation. This was originally proposed for the case $\nu - \nicefrac12 \in \mathbb{N}_0$ \todo{(i.e., $\alpha\in\mathbb{N}$)} and was later extended by \cite{bolin2020numerical, bolin2024covariance} to general $\nu>0$, where the authors also derived explicit rates of convergence of the approximation in terms of the finite element mesh width. The approach is computationally efficient, but has the disadvantage that the SDE has certain boundary conditions which makes the approximation converge to a non-stationary covariance which is only similar to the Mat\'ern covariance away from the boundary of the computational domain. The rate of convergence is better than that of the random Fourier features method, but does not decrease exponentially fast.
Other approaches, which are generally applicable, but without theoretical rates of convergence of the covariance function approximation are covariance tapering \citep{furrer2006covariance}, Vecchia approximation  \citep{vecchia1988, gramacy2015local, Datta2016}, low rank methods \citep{higdon2002space, Cressie2008}, and multiresolution approximations \citep{Nychka2015}.  
The state space methods have also been extended to general  $\nu>0$ in \cite{karvonen2016approximate} and \cite{tronarp2018mixture} through spectral transformation methods and the approximation of the Matérn kernel by a finite scale mixture of squared exponential kernels, respectively. However, the accuracy the resulting approximated covariance function has been demonstrated only through numerical experiments and no theoretical analysis of the rate has been provided.

In this work, we develop a new method for Gaussian Mat\'ern processes on intervals, which has at most $\mathcal{O}(nm^3\lceil{\alpha\rceil^3})$ computational cost when used for statistical inference, prediction and sampling, where $m$ is the order of the approximation. We prove in Section~\ref{sec:rationalapprox} that the error of the covariance function converges exponentially fast in $m$, which in practice means that $m$ can be chosen very low. Specifically, the error is $O\bigl(\exp(-2\pi\sqrt{\{\alpha\}m})\bigr)$ where $\{\alpha\}$ is the fractional part of $\alpha\notin\mathbb{N}$. If  $\alpha\in\mathbb{N}$, the method is exact and has a cost of $\mathcal{O}(n\lceil{\alpha\rceil^3})$. 
The approach is based on a rational approximation of the spectral density, which thus is similar in spirit to \cite{karvonen2016approximate, Roininen2018}. The difference is, however, that we have theoretical guarantees for the error. Because the method is based on a rational approximation, it can be used in combination with state-space methods for efficient inference. We, however, derive in Section~\ref{sec:inference} a direct representation of the approximated process as a sum of Gaussian Markov random fields with sparse precision (inverse covariance) matrices, which in fact are band matrices. This representation has several advantages, and perhaps the most important is that it can directly be incorporated in general software for Bayesian inference, such as \texttt{R-INLA} \citep{lindgren2015bayesian}. 
Another important feature is that the linear cost is applicable to working with the process and its derivatives jointly. This is useful in several applications that arise in the natural sciences where observations of the derivatives are available \cite[see, e.g.][]{solak2002derivative, padidar2021scaling, de2021high, yang2018sparse} \todo{or when derivatives are of direct interest \citep{Swain2016}}.

To validate the effectiveness and accuracy of the method, we compare it in Section~\ref{sec:numerical} to the covariance-based rational approximation method of \cite{bolin2024covariance} as well as the Vecchia approximations \citep{Datta2016}. \todo{These methods demonstrated superior performance among several alternatives for approximating Gaussian Matérn fields in specific geostatistical test problems, as shown in the comparative study by \citet{hong2023third}.}
We also compare with the state-space approach of \citet{karvonen2016approximate}, the random Fourier features method of \citet{rahimi2007random} and the covariance tapering approach of \citet{Nychka2015}. We show that the method outperforms the alternatives in terms of accuracy for a fixed cost when used for Gaussian process regression. The comparison also includes a principal component analysis (PCA) \citep{wang2008karhunen} approach, which serves as an ``optimal'' low-rank method. As the proposed method outperforms this, it means that it also would outperform any other low-rank method, such as fixed rank kriging \citep{Cressie2008} or process convolutions \citep{higdon2002space}.
Extensions of the method beyond stationary Mat\'ern processes on intervals, as well as concluding remarks, are given in Section~\ref{sec:discussion}.
The proposed method is implemented in the R package \texttt{rSPDE} \citep{rSPDE} available on CRAN, and all code for the comparisons, as well as a Shiny application with further results can be found in \url{https://github.com/vpnsctl/MarkovApproxMatern/}. Proofs and technical details are provided in two technical appendices. 

\section{Exponentially convergent rational approximation}\label{sec:rationalapprox}

The proposed method can be obtained through two equivalent formulations. Either through a rational approximation of the spectral density of the Gaussian process, or through a rational approximation of the covariance operator of the process. The latter formulation enables extensions which we will explore in Section~\ref{sec:discussion}. In this section, we outline the idea through the spectral density approach, which is less technical.

Let $u$ be a centered Gaussian Process on an interval $I\subset \mathbb{R}$ with covariance function \eqref{eq:matern_cov} and $\alpha = \nu + \nicefrac{1}{2} \notin\mathbb{N}$ (which is the case for which no exact and efficient methods exist). This process has spectral density 
$
f_{\alpha}(w) = A\sigma^2(\kappa^2 + w^2)^{-\alpha},
$
where \todo{$A=\sqrt{2}\kappa^{2\nu}\Gamma(\nu+\nicefrac{1}{2})\Gamma(\nu)^{-1}$} \citep{lindgren2012stationary}. 
We define a rational approximation of the process $u$ as a Gaussian process with spectral density
\begin{equation}\label{rational_approx}
f_{m,\alpha}(w) = A\kappa^{-2\alpha}\frac{\sigma^2}{(1+\kappa^{-2}w^2)^{\lfloor \alpha \rfloor}} \frac{P_m(1+\kappa^{-2}w^2)}{Q_m(1+\kappa^{-2}w^2)},
\end{equation}
\todo{where $P_m(x) = \sum_{i=0}^m a_i x^{m-i}$ and $Q_m = \sum_{i=0}^{m} b_i x^{m-i}$} are polynomials derived from the optimal rational approximation of order $m$ for the real-valued function $f(x) = x^{\{\alpha\}}$ on the interval $[0, 1]$, with respect to the supremum norm\todo{, where we recall that $\{\alpha\}$ is the fractional part of $\alpha$. More precisely, we use the approximation $x^{-\{\alpha\}} = \left(x^{-1}\right)^{\{\alpha\}} \approx P_n(x)/Q_n(x)$, where the coefficients $\{a_i\}_{i=0}^m$ and $\{b_i\}_{i=0}^m$ are such that the approximation}
\begin{equation}\label{eq:rat_approx_01}
x^{\{\alpha\}} \approx \frac{\sum_{i=0}^{m} a_i x^{i}}{\sum_{i=0}^m b_i x^{i}}
\end{equation}
is the best with respect to the supremum norm on $[0,1]$.
The coefficients $\{a_i\}_{i = 0}^{m}$ and  $\{b_i\}_{i = 0}^{m}$ in the rational approximation \todo{are unique \cite[Chapter 7.2]{lorentz1996constructive} and} can be obtained via the second Remez algorithm \citep{remez1934determination} or by using the recent, and more stable, BRASIL algorithm \citep{hofreither2021algorithm}. \todo{Note that the polynomials $P_m(x)$ and $Q_m$ in \eqref{rational_approx} are written in terms of $x^{m-i}$ instead of $x^{i}$, using the coefficients from \eqref{eq:rat_approx_01}, because the rational approximation is applied to $(x^{-1})^{\{\alpha\}}$.}

\begin{figure}[t]
  \begin{center}
	\includegraphics[height=0.7\linewidth]{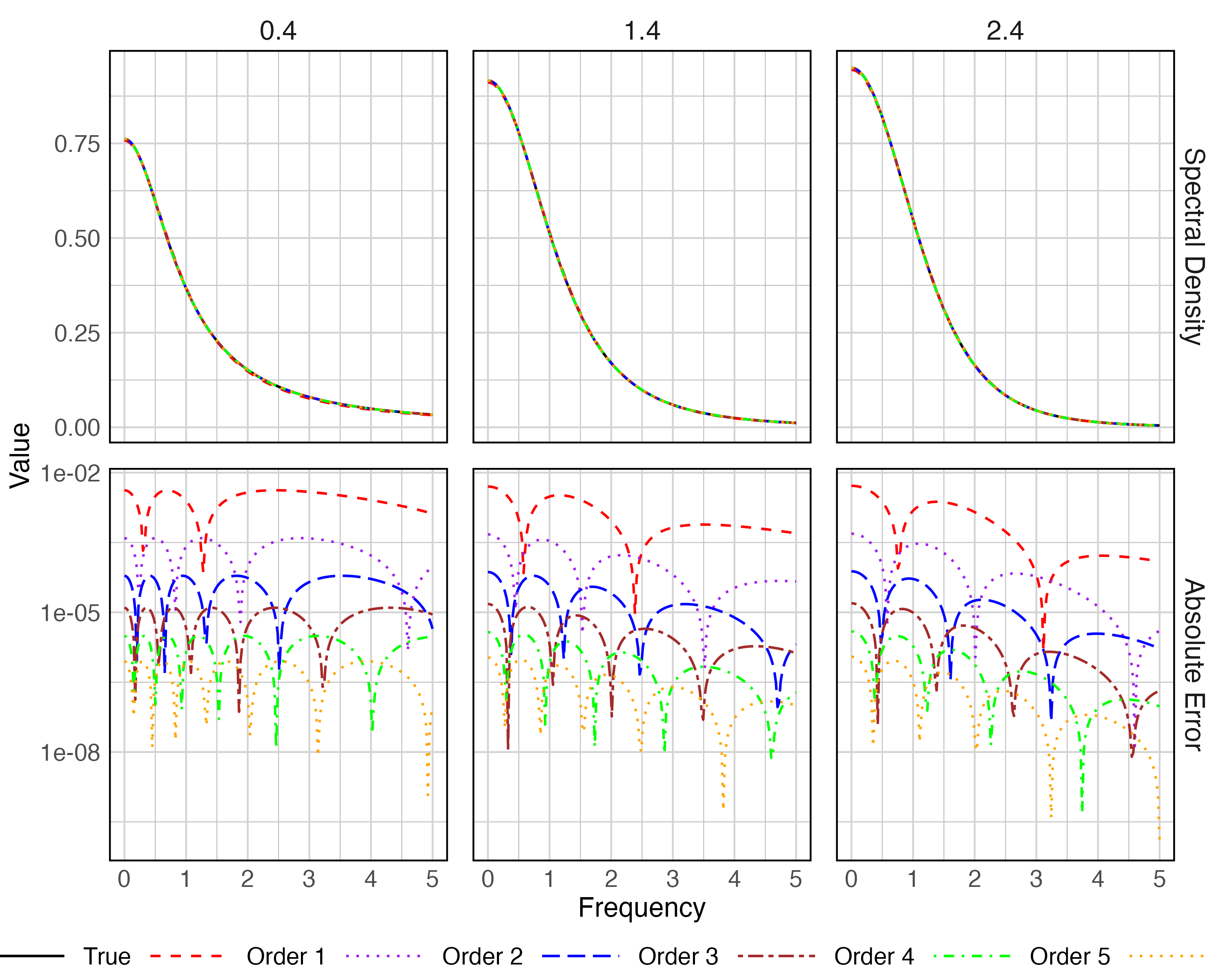}
  \end{center}
  \caption{\todo{True and approximate spectral densities for $\nu=0.4$, $1.4$ and $2.4$, and for different rational approximation orders (top row), and corresponding absolute errors for the rational approximations (bottom row).}}
  \label{fig:spectral_density}
\end{figure}

\todo{The true and approximate spectral densities for $\nu=0.4$, $1.4$ and $2.4$, along with their absolute errors, are shown in Figure \ref{fig:spectral_density} for different orders of rational approximation, where we obtained the coefficients using BRASIL algorithm. Here, we consider $\sigma = 1$ and $\rho = 2$, where $\rho = \nicefrac{\sqrt{8\nu}}{\kappa}$ represents the practical correlation range. We can observe that the approximation is very good even for low orders, with an almost perfect match already for order 2.}

\todo{Let $r_{m,\alpha}$ be the corresponding covariance function obtained from the rational approximation of order $m$ of $f_{\alpha}$ given by \eqref{rational_approx}, more precisely, let
\begin{equation}\label{eq:approx_cov_inv_fourier}
  r_{m,\alpha}(t) = \frac{1}{2\pi}\int_{\mathbb{R}} e^{iwt} f_{m,\alpha}(w) \, dw,\quad t\in\mathbb{R}.
\end{equation}
An explicit expression for $r_{m,\alpha}(\cdot)$ will be obtained in the next section.}

\todo{The following result demonstrates that the approximated covariance function $r_{m,\alpha}(\cdot)$ converges exponentially fast to the true Matérn covariance function with respect to both the $L_2(I \times I)$-norm and the supremum norm. For $f \in L_2(I \times I)$, the $L_2(I \times I)$-norm is defined as $\|f\|_{L_2(I \times I)}^2 = \int_{I \times I} |f(s,t)|^2 \, ds \, dt$, where $L_2(I \times I)$ is the space of equivalence classes of real-valued square-integrable functions on $I \times I$. For $f \in C(I \times I)$, the supremum norm is defined as $\|f\|_{C(I \times I)} = \sup_{(s,t) \in I \times I} |f(s,t)|$, where $C(I \times I)$ is the space of continuous real-valued functions on $I \times I$.}

\begin{theorem}\label{ra_bound}
	\todo{Let $r_\alpha$ be the covariance function 
  \begin{equation}\label{eq:matern_cov_2par}
  r_\alpha(s,t;\kappa,\sigma^2) = \varrho(|s-t|;\alpha-\nicefrac{1}{2},\kappa,\sigma^2), \quad s,t,\in\mathbb{R},
  \end{equation}
  where $\varrho(\cdot)$ is the Mat\'ern covariance function
  defined in \eqref{eq:matern_cov} and $\alpha = \nu+\nicefrac12$. Further, let $r_{m,\alpha}$ be the rational approximation given in \eqref{eq:approx_cov_inv_fourier}.} If $\alpha>1/2$, then
	\begin{equation} \label{total_error_bound}
	\todo{\|r_{m,\alpha} - r_\alpha\|_{L_2(I \times I)} \leq 2 \sqrt{\pi} (b-a) A \sigma^2 \kappa^{-2\alpha} \min\{1, M_{\lfloor\alpha\rfloor,\kappa}\} C_{\{\alpha\}} \mathbb{I}_{\alpha \notin \mathbb{N}} e^{-2\pi\sqrt{\{\alpha\} m}}},
	\end{equation}
    \todo{where $\{\alpha\}$ denotes the fractional part of $\alpha$, $M_{n,\kappa} = \kappa \pi \Gamma(2n-1)/(4^{n-1}\Gamma(n)^2),$ $n\geq 1$ and  $M_{0,\kappa} = \infty$.
    Additionally, $C_{\{\alpha\}}\in (0,\infty)$ is a constant that depends only on $\{\alpha\}$.}
	Further, if $\alpha>1$, then a constant $K_\alpha\in (0,\infty)$ \todo{that only depends on $\alpha$} exists such that
	\begin{equation} \label{sup_bound}
	\todo{\|r_{m,\alpha} - r_\alpha\|_{\todo{C(I \times I)}} \leq A \sigma^2\kappa^{-2\alpha}K_\alpha \mathbb{I}_{\alpha \notin \mathbb{N}} e^{-2\pi\sqrt{\{\alpha\} m}}}.
	\end{equation}
\end{theorem}

\todo{\begin{remark}\label{rem:constant_C_alpha}
The constant $C_{\{\alpha\}}$ in \eqref{total_error_bound} comes from the error of the rational approximation 
$$
	\sup_{x\in [0,1]}\left|x^{\{\alpha\}} - \frac{p_m(x)}{q_m(x)}\right| \leq C_{\{\alpha\}} e^{-2\pi \sqrt{\{\alpha\} m}},
$$ 
where $C_{\{\alpha\}}$ is independent of $m$. 
Further, for $\alpha \in (0,1)$, by \citet[Theorem~2]{saff1995asymptotic}, there is an approximate expression for this uniform error that is valid for large $m$:
$$
  \sup_{x\in [0,1]}\left|x^{\{\alpha\}} - \frac{p_m(x)}{q_m(x)}\right| = 4^{{\{\alpha\}}+1} |\sin(\pi{\{\alpha\}})|e^{-2\pi \sqrt{\{\alpha\} m}} (1+ o(1)),\quad\text{as } m\to\infty.
$$
This expression provides an intuition on how the constant $C_{\{\alpha\}}$ behaves.
\end{remark}}


\todo{The theoretical error \eqref{total_error_bound} using the approximate expression of $C_{\{\alpha\}}$ given in Remark \ref{rem:constant_C_alpha} is shown in Figure~\ref{fig:theoretical_error} for $\rho = 2$, $\sigma =1$, and different values of $\nu$ and $m$ on the interval $I=[0,50]$. We can observe that the errors do not decay monotonically as $\{\alpha\}$ increases, but instead, they have a more complex behavior with respect to $\alpha$. However, they decay monotonically as $m$ increases. One should note that the true error bounds may look a bit different compared to those in  Figure \ref{fig:theoretical_error} as they are based on using $4^{{\{\alpha\}}+1} |\sin(\pi{\{\alpha\}})|$ in place of $C_{\{\alpha\}}$. See Figure \ref{fig:covarianceerror} for the actual covariance errors for this example.}

\begin{figure}[t]
  \begin{center}
	\includegraphics[width=0.8\linewidth]{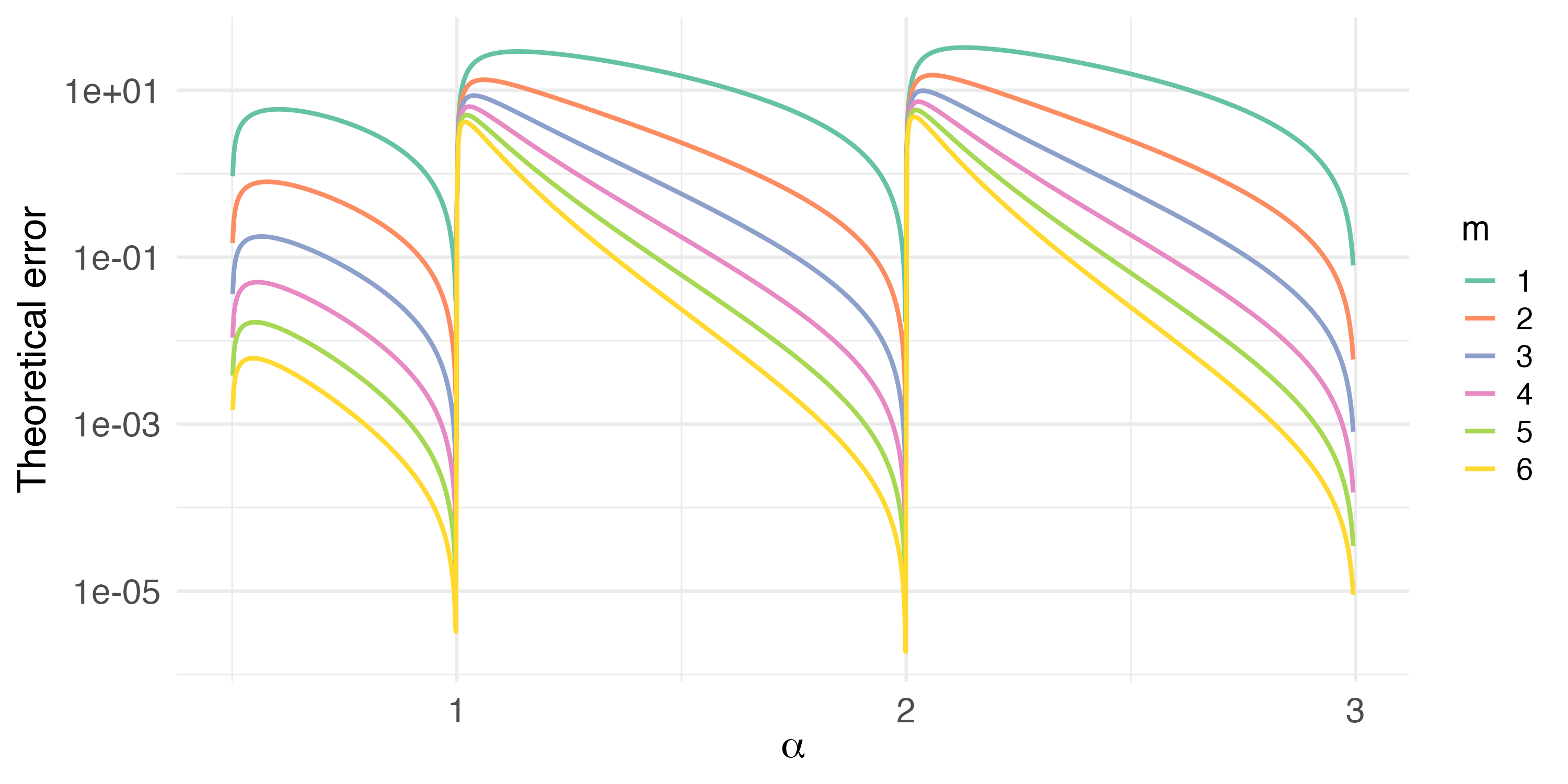}
  \end{center}
  \caption{\todo{Theoretical error bound from \eqref{total_error_bound} in Theorem \ref{ra_bound} using the approximation for $C_{\{\alpha\}}$ given in Remark \ref{rem:constant_C_alpha}.}}
  \label{fig:theoretical_error}
\end{figure}

\section{Linear cost inference}\label{sec:inference}
\todo{Our goal now is to use the rational approximation in \eqref{rational_approx} to obtain a linear cost approximation of the covariance function of the Gaussian process $u$.}
\todo{This approximation will based on a partial fractions decomposition of the rational function in the approximate spectral density. Therefore, we first derive an important property of such partial fractions decompositions.}

\begin{proposition}\label{prp:partial_fractions}
  \todo{Fix $\alpha \in (0,1)$ and $m\in\mathbb{N}$. Let the coefficients $\{a_i\}_{i=0}^m$ and $\{b_i\}_{i=0}^m$ be such that the best rational approximation of $x^{\alpha}$ on $[0,1]$ is given by \eqref{eq:rat_approx_01}. Further, let $P_m(x) = \sum_{i=0}^m a_i x^{m-i}$ and $Q_m = \sum_{i=0}^{m} b_i x^{m-i}$. Then, we have the following partial fractions decomposition of $P_m(x)/Q_m(x)$:
  $$\frac{P_m(x)}{Q_m(x)} = k + \sum_{i=1}^m \frac{c_i}{x - p_i},$$
  where $k, c_i>0$ and $p_i<0$ for $i=1,\ldots,m$.}
\end{proposition}

\begin{remark}
  \todo{Proposition \ref{prp:partial_fractions} fills a theoretical gap left in \cite{bolin2024covariance}, where such a decomposition, along with the signs of $k$, $c_i$, and $p_i$ for $i=1,\ldots,m$, was verified numerically.}
\end{remark}

\todo{In view of Proposition \ref{prp:partial_fractions}}, we can perform a partial fraction decomposition of the rational function $P_m(x)/Q_m(x)$ in \eqref{rational_approx} \todo{to obtain that the spectral density of the approximation is}
\begin{equation}\label{partial_frac}
 \begin{aligned}
f_{m,\alpha}(w)&= A\sigma^2\kappa^{-2\alpha}\left[\frac{k}{(1+\kappa^{-2}w^2)^{\lfloor \alpha \rfloor}}+\sum_{i=1}^{m}c_i\frac{1}{(1+\kappa^{-2}w^2)^{\lfloor \alpha \rfloor}(1+\kappa^{-2}w^2-p_i)}\right]\\
&=:\left[f_{m,0,\alpha}(w)+\sum_{i=1}^{m}f_{m,i,\alpha}(w)\right],
\end{aligned}
 \end{equation}
\todo{where $k,c_i>0$ and $p_i<0$ for $i=1,\ldots,m$, and $A=\sqrt{2}\kappa^{2\nu}\Gamma(\nu+\nicefrac{1}{2})\Gamma(\nu)^{-1}$. Thus, we obtain that $f_{m,\alpha}(\cdot)$ can be decomposed as a sum of valid spectral densities. Hence, we can write $r_{m,\alpha}(\cdot,\cdot)$, given in \eqref{eq:approx_cov_inv_fourier}, as a sum of covariance functions}. 

Let $\varrho_{m,\alpha}(\cdot)$ be defined as $\varrho_{m,\alpha}(t-s) = r_{m,\alpha}(t,s)$. Based on \eqref{partial_frac}, we obtain the following explicit expression of the approximated covariance function $\varrho_{m,\alpha}(\cdot)$.

\begin{proposition}\label{cov_prop}
Let $u$ be a Gaussian process with spectral density \eqref{partial_frac}. Then, it has covariance function 
\begin{equation}\label{covrational}
\varrho_{m,\alpha}(h)=\varrho_{m,0,\alpha}(h)+\displaystyle\sum_{i=1}^{m}\varrho_{m,i,\alpha}(h),
\end{equation}
where 
$$
\varrho_{m,0,\alpha}(h)=
k\sigma^2\cdot\begin{cases}
\frac{c_{\alpha}\sqrt{4\pi}}{\kappa} 1_{[h=0]} &0<\alpha<1,\\
\varrho\left(h; \lfloor \alpha \rfloor -\frac{1}{2}, \kappa, \frac{c_{\alpha}}{c_{\lfloor\alpha\rfloor}}\right)  & \alpha\geq 1,
\end{cases}
$$
and
$$
\varrho_{m,i,\alpha}(h)= c_i\sigma^2\cdot
\begin{cases}
\varrho\left(h;\frac{1}{2}, \kappa_i, \frac{c_{\alpha}\sqrt{\pi}}{\sqrt{1-p_i}}\right)
& 0<\alpha<1,\\
\frac{1}{p_i^{\lfloor \alpha \rfloor}}\varrho\left(h;\frac{1}{2}, \kappa_i,\frac{c_{\alpha}\sqrt{\pi}}{\sqrt{1-p_i} }\right) -
\displaystyle\sum_{j=1}^{\lfloor \alpha \rfloor}
\frac{1}{p_i^{\lfloor \alpha \rfloor+1-j}}\varrho\left(h; j - \frac{1}{2}, \kappa, 
\frac{c_{\alpha}}{c_j }\right) & \alpha\geq 1.
\end{cases}
$$
Here $c_{a} := \nicefrac{\Gamma(a)}{\Gamma(a-\nicefrac{1}{2})}$, $\kappa_i = \kappa\sqrt{1-p_i}$, and $\varrho$ is the Mat\'ern covariance \eqref{eq:matern_cov}.
\end{proposition}

Since the spectral density in \eqref{partial_frac} is a sum of valid spectral densities $f_{m,0,\alpha}$ and $f_{m,i,\alpha}$ for $i = 1, \ldots, m$, a Gaussian process with the spectral density given by \eqref{partial_frac} can be expressed as a sum of independent Gaussian processes
\begin{equation}\label{approxsol}
    u(x) = u_0(x) + u_1(x) + \dotsb + u_m(x),
\end{equation}
where $u_0$ has spectral density $f_{m,0,\alpha}$ and each $u_i$ has spectral density $f_{m,i,\alpha}$ for $i = 1, 2, \dotsc, m$. All these spectral densities are reciprocals of polynomials, implying that each process $u_i$, for $i = 0, \ldots, m$, is a Gaussian Markov process \cite[Theorem 10.1]{pitt1971markov}.
Specifically, $u_0$ is a Markov process of order $\max(\lfloor \alpha \rfloor, 1)$, and it is $\max(\lfloor \alpha \rfloor - 1, 0)$ times differentiable in the mean-squared sense. Moreover, for $i > 0$, each $u_i$ is a Markov process of order $\lceil \alpha \rceil$, and it is $\lfloor \alpha \rfloor$ times differentiable in the mean-squared sense. As a result, the multivariate process $\mv{u}_0(t) = \left( u_0(t), u_0'(t), \ldots, u_0^{(\max(\lfloor \alpha \rfloor - 1, 0))}(t) \right)$ is a first-order Markov process with a multivariate covariance function
$$
\mv{r}_0 : \mathbb{R} \times \mathbb{R} \rightarrow \mathbb{R}^{\lfloor \alpha \rfloor \times \lfloor \alpha \rfloor}, \quad \mv{r}_0(s,t) = \left[\frac{\partial^{i-1}}{\partial s^{i-1}}\frac{\partial^{j-1}}{\partial t^{j-1}}\varrho_{m,0,\alpha}(s-t)\right]_{i,j\in \{1,\ldots, \max(\lfloor\alpha\rfloor,1)\}}.
$$
Similarly, the multivariate process $\mv{u}_i(t) = \left( u_0(t), u_0'(t), \ldots, u_0^{(\lfloor \alpha \rfloor)}(t) \right)$, for $i \in \{1, \ldots, m\}$, is a first-order Markov process with a multivariate covariance function  
\[
\mv{r}_i : \mathbb{R} \times \mathbb{R} \rightarrow \mathbb{R}^{\lceil \alpha \rceil \times \lceil \alpha \rceil}, \quad 
\mv{r}_i(s,t) = \left[ \frac{\partial^{k-1}}{\partial s^{k-1}} \frac{\partial^{\ell-1}}{\partial t^{\ell-1}} \varrho_{m,i,\alpha}(s - t) \right]_{k, \ell \in \{1, \ldots, \lceil \alpha \rceil\}}.
\]
Since these multivariate processes are first-order Markov, we can derive the following result regarding their finite-dimensional distributions.

\begin{proposition}\label{ra_prec}
Consider a set of \todo{unique} locations $t_1, \ldots, t_n \in I$. For $j = 1, \ldots, n$, define the vectors  
\[
\mv{u}_{0,j} = \left[ u_0(t_j), u_0'(t_j), \ldots, u_0^{(\max(\lfloor \alpha \rfloor - 1, 0))}(t_j) \right], 
\quad \text{and} \quad 
\mv{u}_{i,j} = \left[ u_i(t_j), u_i'(t_j), \ldots, u_i^{(\lfloor \alpha \rfloor)}(t_j) \right], 
\]  
for $i \in \{1, \ldots, m\}$. Then, the concatenated vector $\mv{u}_i = \left[ \mv{u}_{i,1}, \ldots, \mv{u}_{i,n} \right]$ is a centered Gaussian random variable with a block tridiagonal precision matrix given by  
\begin{equation}\label{precmat}
\mv{Q}_{i} = 
\begin{bmatrix*}[l]
\mv{Q}_{1,1} & \mv{Q}_{1,2} & & & & \\
\mv{Q}_{2,1} & \mv{Q}_{2,2} & \mv{Q}_{2,3} & & & \\
& \mv{Q}_{3,2} & \mv{Q}_{3,3} & \mv{Q}_{3,4} &  &  \\
& & \ddots &  \ddots   &  \ddots &  \\
& & & \mv{Q}_{n-1,n-2} & \mv{Q}_{n-1,n-1} &  \mv{Q}_{n-1,n} \\
\phantom{\mv{Q}_{n-1,n-1}} & \phantom{\mv{Q}_{n-1,n-1}} & \phantom{\mv{Q}_{n-1,n-1}} & \phantom{\mv{Q}_{n-1,n-1}} & \mv{Q}_{n,n-1} & \mv{Q}_{n,n}
\end{bmatrix*},
\end{equation}
for $i = 0, 1, \ldots, m$. 
The first set of blocks, $\mv{Q}_{1,1}, \mv{Q}_{1,2},$ and $\mv{Q}_{2,1}$, are obtained as  
\begin{equation}\label{eq:qsolve1}
\begin{bmatrix}
\mv{Q}_{1,1} & \mv{Q}_{1,2} \\
\mv{Q}_{2,1} & \mv{M}
\end{bmatrix} = 
\begin{bmatrix}
\mv{r}_i(t_1, t_1) & \mv{r}_i(t_1, t_2) \\
\mv{r}_i(t_2, t_1) & \mv{r}_i(t_2, t_2)
\end{bmatrix}^{-1},
\end{equation}
where $\mv{M}$ is a dummy variable that is discarded. 
For $j \in \{2, \ldots, n - 1\}$, the subsequent blocks $\mv{Q}_{j,j}, \mv{Q}_{j,j+1}, \mv{Q}_{j+1,j},$ and $\mv{Q}_{j+1,j+1}$ are obtained as  
\begin{equation}\label{eq:qsolve2}
\begin{bmatrix}
\mv{M}_1 & \mv{M}_2 & \mv{0} \\
\mv{M}_2 & \mv{Q}_{j,j} & \mv{Q}_{j,j+1} \\
\mv{0} & \mv{Q}_{j+1,j} & \mv{Q}_{j+1,j+1}
\end{bmatrix} = 
\begin{bmatrix}
\mv{r}_i(t_{j-1}, t_{j-1}) & \mv{r}_i(t_{j-1}, t_j) & \mv{r}_i(t_{j-1}, t_{j+1}) \\
\mv{r}_i(t_j, t_{j-1}) & \mv{r}_i(t_j, t_j) & \mv{r}_i(t_j, t_{j+1}) \\
\mv{r}_i(t_{j+1}, t_{j-1}) & \mv{r}_i(t_{j+1}, t_j) & \mv{r}_i(t_{j+1}, t_{j+1})
\end{bmatrix}^{-1},
\end{equation}
where $\mv{M}_1$ and $\mv{M}_2$ are dummy variables that are also discarded.
\end{proposition}

The inverses in \eqref{eq:qsolve1} and \eqref{eq:qsolve2} can be computed analytically. However, since these matrices have sizes $2 \max(\lfloor \alpha \rfloor, 1) \times 2 \max(\lfloor \alpha \rfloor, 1)$ and $3 \max(\lfloor \alpha \rfloor, 1) \times 3 \max(\lfloor \alpha \rfloor, 1)$, respectively, for the process $\mv{u}_0$, and sizes $2 \lceil \alpha \rceil \times 2 \lceil \alpha \rceil$ and $2(\lceil \alpha \rceil + 1) \times 2(\lceil \alpha \rceil + 1)$, respectively, for the process $\mv{u}_i$ with $i > 0$, their numerical inversion is computationally trivial. As a result, the benefit of deriving analytical expressions is negligible.

\begin{remark}
If $\alpha\in\mathbb{N}$, there is no need for a rational approximation since the process itself is Markov. In this case, we can derive the precision matrix for the process and its derivatives using the same strategy as for $\mv{u}_i$ in Proposition~\ref{ra_prec} but where the covariance function $\mv{r}_i$ is replaced by the multivariate covariance function of $u$ and its derivatives. We do not go into more details as there already are exact methods for $\alpha\in\mathbb{N}$. It should, however, be noted that the costs of using this exact method are the same as those we discuss below with $m=1$.    
\end{remark}

We now demonstrate how these expressions can be utilized for computationally efficient sampling, inference, and prediction. Let $t_1, \ldots, t_n$ be a set of locations in $I$, and define the vector {$\mv{U}_i = \left[ \mv{u}_i(t_1), \ldots, \mv{u}_i(t_n) \right]$} for $i = 0, \ldots, m$. Since these multivariate processes are independent, the vector  
$
\bar{\mv{U}} = \left[ \mv{U}_0^\top, \ldots, \mv{U}_m^\top \right]^\top
$
is a centered multivariate Gaussian random variable of dimension  
$
N = n \left( m \lceil \alpha \rceil + \max(\lfloor \alpha \rfloor, 1) \right),
$
with a block diagonal precision matrix $\mv{Q} = \operatorname{diag}(\mv{Q}_0, \mv{Q}_1, \ldots, \mv{Q}_m)$, where each block is obtained using Proposition~\ref{ra_prec}.

Because $\mv{Q}$ is a band matrix, the following result provides the computational cost of computing its Cholesky factor.

\begin{proposition}\label{prop:chol}
Computing the Cholesky factor $\mv{R}$ of $\mv{Q}$, $\mv{Q} = \mv{R}^\top\mv{R}$, requires 
$\mathcal{O}(nm\lceil\alpha\rceil^3)$ floating point operations, given that $\alpha \ll n$. 
Further, solving $\mv{Y} = \mv{R}^{-1}\mv{X}$ for some vector $\mv{X} \in \mathbb{R}^{N}$ requires $\mathcal{O}(nm\lceil\alpha\rceil^2)$ floating point operations.
\end{proposition}

Thus, the computational cost for sampling $\bar{\mv{U}}$ by first computing the Cholesky factor $\mv{R}$ of $\mv{Q}$ and then solving $\bar{\mv{U}} = \mv{R}^{-1}\mv{Z}$ for a vector $\mv{Z}$ with independent standard Gaussian elements can be done in $\mathcal{O}(nm\lceil\alpha\rceil^{3})$ computational cost. Moreover, introduce the sparse matrix $\mv{A} = [\mv{A}_0,\ldots, \mv{A}_m]$ where $\mv{A}_0$ is the sparse $n\times n\max(\lfloor \alpha \rfloor,1)$ matrix that extracts the values $u_0(t_1),\ldots, u_0(t_n)$ from the vector $\mv{U}_0$ (i.e., a matrix where each row has one 1  and the rest of the values equal to 0) and where $\mv{A}_i$ for $i>0$ is the $n\times n\lceil \alpha \rceil$ matrix that extracts the values $u_i(t_1),\ldots, u_i(t_n)$ from the vector $\mv{U}_i$. We then have that $\mv{u} = [u(t_1), \ldots, u(t_n)]^\top = \mv{A}\bar{\mv{U}}$. Hence, it follows that 
$\mv{u} \sim \mathcal{N}(0,\mv{A}\mv{Q}^{-1}\mv{A^{\top}})$,
and we can sample $\mv{u}$ in $\mathcal{O}(nm\lceil\alpha\rceil^{3})$ cost by first sampling $\bar{\mv{U}}$ and then computing $\mv{u} = \mv{A}\bar{\mv{U}}$.

Partitioning $\bar{\mv{U}} = (\bar{\mv{U}}_A^{\top}\, \bar{\mv{U}}_{B}^{\top})^{\top}$ for some  $A\cup B = \{1,\ldots, N\}$ with $A\cap B = \emptyset$, 
we have that 
$\bar{\mv{U}}_A | \bar{\mv{U}}_B \sim \mathcal{N}(-\mv{Q}_{A,A}^{-1}\mv{Q}_{A,B}\bar{\mv{U}}_B,\mv{Q}_{A,A}^{-1})$. This means that conditional distributions also can be computed efficiently, and in particular, the conditional mean $\mv{\mu}_{A|B} = -\mv{Q}_{A,A}^{-1}\mv{Q}_{A,B}\bar{\mv{U}}_B$ can be computed in $\mathcal{O}(|A|(2\lfloor\alpha\rfloor+1)^2)$ computational cost. 

These costs do not include the cost of building $\mv{Q}$; however, it turns out that one can directly construct an LDL factorization $\mv{Q}_i = \mv{L}_i^\top\mv{D}_i\mv{L}_i$ at a slightly lower computional cost than of computing $\mv{Q}_i$ through the following result.  

\begin{proposition}\label{ra_chol}
Using the same notation as in Proposition~\ref{ra_prec}, the precision matrix of  $\mv{u}_i$ can be constructed as $\mv{Q}_i = \mv{L}_i^\top\mv{D}_i\mv{L}_i$. Here $\mv{D}_i$ is a diagonal matrix with positive diagonal entries and $\mv{L}_i$ is a lower triangular block matrix with ones on the main diagonal. Specifically, $\mv{L}_i$ has the form 
\begin{equation}\label{Lmat}
\mv{L}_i = 
\begin{bmatrix*}[l]
\mv{L}_{1,1}  &              &  	            &       &    \\
\mv{L}_{2,1}  & \mv{L}_{2,2} &                  &       &    \\
              & \mv{L}_{3,2} & \mv{L}_{3,3}     &        &    \\
              &              & \ddots           & \ddots &    \\
              &              &    & \hspace{-0.5cm}\mv{L}_{n,n-1}&  \mv{L}_{n,n} 
  \end{bmatrix*},
\end{equation}
Here all blocks are of the same size as $\mv{u}_{i,j}$ and the matrices $\mv{L}_i$ and $\mv{D}_i$ can be constructed in $\mathcal{O}(n\max(\lfloor\alpha\rfloor,1)^4)$ and $\mathcal{O}(n\lceil\alpha\rceil^4)$ computational cost for $i=0$ and $i>0$, respectively, through the method in Appendix~\ref{ldlsection}.
\end{proposition}

It should be noted that the costs $\mathcal{O}(n\max(\lfloor\alpha\rfloor,1)^4)$ and $\mathcal{O}(n\lceil\alpha\rceil^4)$ can likely be reduced slightly by taking advantage of that several redundant calculations are performed in the construction. 
Forming $\mv{L} = diag(\mv{L}_0,\mv{L}_{1},\ldots,\mv{L}_{m})$ and $\mv{D} = diag(\mv{D}_0,\mv{D}_{1},\ldots,\mv{D}_{m})$,  where the blocks are obtained by using Proposition \ref{ra_chol}, we can now, for example, sample $\bar{\mv{U}}$ as $\bar{\mv{U}} = \mv{L}^{-1}\mv{D}^{-\nicefrac12}\mv{Z}$ at $\mathcal{O}(nm\lceil\alpha\rceil^2)$
computational cost.

Next, consider the situation of a Gaussian process regression where the stochastic process is observed under Gaussian measurement noise. That is, suppose that we have observations $\mv{y}=[y_1,y_2,\dotsc,y_n]$ obtained as
$
y_i|u(\cdot)\sim \mathcal{N}(u(t_i),\sigma_e^2).
$
Then, $\mv{y}|\bar{\mv{U}}\sim \mathcal{N}(\mv{A}\bar{\mv{u}}, \sigma_e^{2}\textbf{I})$ and ${\mv{\bar{\mv{U}}}\sim \mathcal{N}(0,\mv{Q}^{-1})}$, where the matrix $\mv{A}$ and the precision matrix $\mv{Q}$ are same as defined above. The goal is typically to estimate the latent process by computing the posterior mean $\mv{\mu}_{\mv{u}|\mv{y}} = \mathbb{E}(\mv{u}|\mv{y})$ of the vector $\mv{u} = (u(t_1),\ldots, u(t_n))^\top$.

By standard results for latent Gaussian models, we obtain $\bar{\mv{U}}|\mv{y}\sim \mathcal{N}(\mv{\mu}_{\bar{\mv{U}}|\mv{y}}, \mv{Q}^{-1}_{\bar{\mv{U}}|\mv{y}})$, where
$$
\mv{\mu}_{\bar{\mv{U}}|\mv{y}}=\frac{1}{\sigma_e^2} \mv{Q}^{-1}_{\bar{\mv{U}}|\mv{y}}\mv{A}^{\top}\mv{y}\quad\text{and}\quad\mv{Q}_{\bar{\mv{U}}|\mv{y}}=\mv{Q}+\frac{1}{\sigma_e^2}\mv{A}^{\top}\mv{A},
$$
and $\mv{\mu}_{\mv{u}|\mv{y}} = \mv{A}\mv{\mu}_{\bar{\mv{U}}|\mv{y}}$.
The sparsity structure of $\mv{Q}_{\bar{\mv{U}}|\mv{y}}$ is different from  $\mv{Q}$; however, we still have linear cost for computing its Cholesky factor through the following proposition.

\begin{proposition}\label{prop:chol_post}
Computing the Cholesky factor $\mv{R}_{\bar{\mv{U}}|\mv{y}}$ of $\mv{Q}_{\bar{\mv{U}}|\mv{y}}$ requires 
$\mathcal{O}(nm^3\lceil\alpha\rceil^3)$ floating point operations, given that $m,\alpha \ll n$. 
Further, solving $\mv{Y} = \mv{R}^{-1}\mv{X}$ for some vector $\mv{X} \in \mathbb{R}^{N}$ requires $\mathcal{O}(nm^2\lceil\alpha\rceil^2)$ floating point operations.
\end{proposition}

Thus, the cost of obtaining the Cholesky factor $\mv{R}_{\bar{\mv{U}}|\mv{y}}$ and computing the posterior mean $\mu_{\bar{\mv{U}}|\mv{y}}$ is $\mathcal{O}(nm^3\lceil\alpha\rceil^3)$. 
Finally, the log-likelihood of $\mv{y}$ is
\begin{equation*}
2\ell(\mv{y})=\log |\mv{Q}|-2n\log(\sigma_e)-\log|\mv{Q}_{\bar{\mv{U}}|\mv{y}}|-\mu_{\bar{\mv{U}}|\mv{y}}^{\top}\mv{Q}\mu_{\bar{\mv{U}}|\mv{y}} -\frac{1}{\sigma_e^2}\|\mv{y}-\mv{A}\mu_{\bar{\mv{U}}|\mv{y}}\|^2-\log(2\pi).
\end{equation*}
Computing this requires computing the Cholesky factors of $\mv{Q}$ and $\mv{Q}_{\bar{\mv{U}}|\mv{y}}$, after which we obtain the log-determinants and the posterior mean in linear cost. Therefore, the total cost for evaluating the log-likelihood, and thus, performing likelihood-based inference is $\mathcal{O}(nm^3\lceil\alpha\rceil^3)$.
To conclude, all relevant tasks for applying the proposed rational approximation in statistical inference require at most $\mathcal{O}(nm^3\lceil\alpha\rceil^3)$ computational cost, and are thus linear in $n$.

\begin{remark}\label{rationalstate}
Because the rational approximation of the previous section results in a Gaussian process with a spectral density that is a rational function, an alternative to the Markov representations above is to directly apply the state space methods of \citet{karvonen2016approximate} to perform inference and prediction at a linear cost $\mathcal{O}(M^3n)$, where $M$ is the approximation order for the state space method. However, the Markov approximation can directly be used in general Bayesian inference software such as R-INLA \citep{rue2009}, which is not possible for the state space methods. 
\end{remark}

\section{Numerical results}\label{sec:numerical}

In this section, we illustrate the performance and accuracy of the proposed rational approximation method by comparing it with several alternative approaches. Specifically, we compare our method (referred to as \textit{Rational Approximation}) with the state-space method of \cite{karvonen2016approximate}, the nearest-neighbor Gaussian process approximation (referred to as \textit{nnGP}) of \cite{Datta2016}, the random Fourier features method (referred to as \textit{Fourier}) from \cite{rahimi2007random}, a principal component analysis approach (referred to as \textit{PCA}) to serve as a lower bound for low-rank methods, the tapering method (referred to as \textit{Taper}) of \cite{furrer2006covariance}, and the covariance-based rational SPDE approach (referred to as \textit{FEM}) from \cite{bolin2024covariance}.

We compare the accuracy of the proposed method with the alternative methods by carrying out three tasks. First, we measure and compare the accuracy in terms of the accuracy of the covariance function, and then consider the cost of Gaussian process regression and measure the quality in terms of the accuracy of the posterior mean $\mv{\mu}_{\mv{u}|\mv{y}}$ \todo{and posterior standard deviation $\mv{\sigma}_{\mv{u}|\mv{y}}$}. Finally, we compare the accuracy of the process approximation as a whole by computing \todo{Kullback--Leibler (KL) divergences and} coverage probabilities of joint confidence bands in a Gaussian process regression. In all these cases, we consider the accuracy for a fixed computational cost (as explained below) to make the comparison completely fair. Further, all methods were implemented 
in R, using the same methods for sparse matrices, matrix solves, and other tasks. The results were obtained using a MacBook Pro Laptop with an M3 Max processor and 128Gb of memory, without using any parallel computations to make the comparison as fair as possible.

\begin{table}[t]
\centering
\caption{Asymptotic costs for different methods. For nnGP, $m$ is the number of neighbors and for the low rank methods, $m$ is the rank. For the state-space method, the ``preconditioned'' approximation of \cite{karvonen2016approximate} is used in combination with our Markov approach, see Remark~\ref{rationalstate}. 
All costs are in ``Big O'' assuming $n$ is much larger than $\alpha$ and $m$. For FEM, $N$ is the number of mesh nodes.}
\label{tab:proposed_methods}
 \small 
 \setlength{\tabcolsep}{5pt}
\begin{tabular} {llll}
\toprule
Method & Construction & Sampling & Prediction \\
\hline
Proposed & $nm\lceil \alpha \rceil^4$ & $nm\lceil \alpha \rceil^2$   & $nm^3\lceil \alpha \rceil^3$\\
state-space \citep{karvonen2016approximate} & $n(m+\lfloor \alpha \rfloor)\lceil\alpha\rceil^4$ & $n(m+\lfloor \alpha \rfloor)\lceil\alpha\rceil^2$   & $n(m+\lfloor \alpha \rfloor)^3\lceil\alpha\rceil^3$\\
nnGP \citep{Datta2016} & $n m^3$ & $n m$ & $n m^2$ \\
Fourier \citep{rahimi2007random} & $nm$ & $nm$ & $nm^2$\\
PCA \citep{wang2008karhunen} & $n^3$ & $nm$ & $nm^2$\\
Tapering \citep{furrer2006covariance} & $nm$ & $nm$ & $nm^2$ \\ 
FEM \citep{bolin2024covariance} & $Nm\lceil \alpha \rceil + n$ & $Nm\lceil \alpha \rceil^2 + n$ & $Nm^2\lceil \alpha \rceil^2 + n$ \\ 
\bottomrule
\end{tabular}
\label{methods_cost}
\end{table}

\subsection{Setup for the first comparison}\label{samp_cov_error}
We first consider the case of a dataset with 5000 observations, $\mv{y} = [y_1, y_2, \dotsc, y_n]$, \todo{with the observation locations being evenly spaced} over the interval \todo{$I = [0, 50]$}. Each observation is generated as  
$y_i \mid u(\cdot) \sim \mathcal{N}(u(t_i), \sigma_e^2),$
where $t_i$ are the observation locations, and $u$ is a centered Gaussian process with the Matérn covariance function given by \eqref{eq:matern_cov}. We set $\sigma = 1$ (as it is merely a scaling parameter) \todo{and explore two different noise levels: $\sigma_e = 0.1$ and $\sigma_e = \sqrt{0.1}$. Additionally, we} vary the \todo{smoothness parameter} $\nu$ over the interval $(0, 2.5)$. For each value of $\nu$, we choose $\kappa = \sqrt{2\nu}$, ensuring that the practical correlation range, $\rho = \nicefrac{\sqrt{8\nu}}{\kappa}$, remains fixed at 2.

This choice of range, along with the interval length, ensures numerical stability across all methods compared. Specifically, a practical correlation range of 2 was the largest range for which all methods, including the most sensitive (nnGP), remained stable. 
While a practical correlation range of 2 on a domain of size 50 may seem small, stability and accuracy of the predictions depend on the number of observations within the correlation range at any given point, rather than the size of the domain itself. Other ranges, and other numbers of observation and prediction locations, \todo{are explored in the accompanying Shiny app.
For example, one scenario considers $10000$ evenly spaced prediction locations and $5000$ observation locations randomly selected without replacement from the prediction locations.}

The goal is to compute the posterior mean $\mv{\mu}_{\mv{u}|\mv{y}}$ with elements $(\mv{\mu}_{\mv{u}|\mv{y}})_j = \mathbb{E}(u(p_j)|\mv{y})$\todo{, and the posterior standard deviation $\mv{\sigma}_{\mv{u}|\mv{y}}$ with elements $(\mv{\sigma}_{\mv{u}|\mv{y}})_j = \sqrt{\Var(u(p_j)|\mv{y})}$}, where $p_j$ are the prediction locations. To make the comparison simple, we initially choose $p_j = t_j$ evenly spaced in the interval.
Because the locations are evenly spaced, we could also include the Toeplitz method of \cite{supergauss} in the comparison. However, we do not include that here as we only consider generally applicable methods. 

\subsection{Calibration of the computational costs}
The asymptotic costs of the different methods are summarized in Table~\ref{tab:proposed_methods}. However, we calibrate the methods to have the same total runtime for assembling all matrices (the construction cost) and computing the posterior mean (the prediction cost). This ensures fair comparisons based on actual performance rather than theoretical complexity, which can be affected by the constants in the cost expression and the size of the study.
Specifically, for a  given set of parameters $(\kappa, \sigma, \nu)$, and a fixed value of $m$ for the proposed method, we calibrate the values of $m$ for the other methods to ensure that the total computation time is the same. The total prediction times were averaged over 100 samples to obtain the calibration. The final calibration results are shown in Table~\ref{tab:calcov}. We now provide some details regarding this calibration procedure.

\begin{table}[t]
\centering
\caption{The choice of $m$ for the different methods that result in equal computational cost as for the proposed method.  The cost for the Fourier method is the same as for PCA, SS denotes the state-space method and Tap the tapering method.}
\label{tab:calcov}
 \small 
 \setlength{\tabcolsep}{5pt}
\begin{tabular} {lllllllllllll}
\toprule
& \multicolumn{4}{c}{$\nu<0.5$} & \multicolumn{4}{c}{$0.5<\nu<1.5$} & \multicolumn{4}{c}{$1.5 < \nu<2.5$}\\
 \cmidrule(r){2-5}  \cmidrule(r){6-9}  \cmidrule(r){10-13} 
m & nnGP & PCA & SS & Tap & nnGP & PCA & SS & Tap & nnGP & PCA & SS & Tap \\
\hline
2 & 1 & 308 & 2 & 1 & 2 & 473 & 1 & 2 & 30 & 810 & 1 & 210 \\
3 & 1 & 355 & 3 & 1 & 13 & 561 & 2 & 3 & 37 & 945 & 1 & 342\\
4 & 1 & 406 & 4 & 1 & 21 & 651 & 3 & 62 & 45 & 1082 & 2 & 376\\
5 & 1 & 433 & 5 & 1 & 27 & 708 & 4 & 124 & 51 & 1205 & 3 & 405\\
6 & 1 & 478 & 6 & 1 & 31 & 776 & 5 & 166 & 54 & 1325 & 4 & 501\\
\bottomrule
\end{tabular}
\end{table}

\begin{table}[t]
\centering
\caption{Number of mesh nodes $N$ for the FEM method for different values of $m$. }
\bigskip
\label{tab:fem}
 \small 
 \setlength{\tabcolsep}{5pt}
\begin{tabular} {llllll}
\toprule
m & 2 & 3 & 4 & 5 & 6\\
\hline
$\nu<0.5$ &  10495& 15494& 20493& 20493& 20493\\
$0.5<\nu<1.5$ &  25492 & 35490 &40489& 45488& 50487\\
$1.5 < \nu<2.5$ & 25492 & 20493& 15494& 15494& 10495\\
\bottomrule
\end{tabular}
\end{table}

Some of the methods have costs which depend on the smoothness parameter $\nu$. The calibration is therefore done separately for the ranges $0 < \nu < 0.5$, $0.5 < \nu < 1.5$, and $1.5 < \nu < 2.5$. 
%
%
For the taper method, the taper function was chosen as in \cite{furrer2006covariance} depending on the value of $\nu$, and the taper range was chosen so that each observation, on average, had $m$ neighbors within the range, and the value of $m$ was then chosen to ensure that the total computational cost matches that of the rational approximation.
For $\nu < 0.5$, the calibration was not possible for nnGP and taper, because the rational approximation remained faster even with $m = 1$. For these cases, we set $m = 1$. 

Given the number of basis functions, the PCA and Fourier methods have the same computational cost for prediction. Thus, the value of $m$ for the Fourier method (the number of basis functions) was set to match the value obtained for the PCA method. The PCA method was calibrated disregarding the construction cost, which is equivalent to assuming that we know the eigenvectors of the covariance matrix. This is not realistic in practice, but makes the method act as a theoretical lower bound for any low-rank method.

As described in Remark~\ref{rationalstate}, the state-space method provides an alternative Markov representation for which we could use the same computational methods as for the proposed method. Its value of $m$ was therefore chosen according to Table~\ref{tab:proposed_methods} as $m - \lfloor \alpha \rfloor$.  

To minimize boundary effects of the FEM method, we extended the original domain $[0, 50]$ to $[-4\rho, 50 + 4\rho]$, where $\rho$ is the practical correlation range. This extension ensures accurate approximations of the Matérn covariance at the target locations. \todo{We refer the reader to \citet[Theorem 3.2]{boundary_effect_WM} and \citet[Proposition 2]{bolin2024covariance}  for a theoretical justification of this commonly used procedure}. Because the FEM method uses the same type of rational approximation as the propose method, we fixed the value of $m$ for the FEM method to be equal to $m$ for the proposed method. The calibration was then instead performed on the finite element mesh. Specifically, a mesh with $N = kn + 500 - (k+3)$ locations in the extended domain used, where $500$ locations were in the extensions and $kn$ locations in the interior $[0,50]$, and the $-(k+3)$ term appears to ensure that the regular mesh contains the observation locations. These $kn$ locations were chosen equally spaced to include the observation locations and $k\in\mathbb{N}$ was calibrated to make the total computational cost match that of the proposed method for $\nu < 1.5$, and for  $1.5 < \nu < 2.5$ it was chosen as the largest values that yielded a stable prediction, as the value which matched the computational cost yielded unstable predictions. The resulting values of $N$ are shown in Table~\ref{tab:fem}.

\subsection{Covariance errors}\label{sec:covariance_errors}
We measure the quality of the approximation by calculating the $L_2$ error and the $L_{\infty}$ error of the respective covariance function approximation, computed as 
\begin{equation*}
	\|r_\alpha-\hat{r}_\alpha\|_{L_2(I\times I)}^2\approx\frac{1}{n^2}\sum_{i,j=1}^{n}(r(t_i,t_j)-\hat{r}(t_i,t_j))^2, \,\,\,
 \|r-\hat{r}\|_{C(I\times I)}\approx\max_{i,j}|r(t_i,t_j)-\hat{r}(t_i,t_j)|,
\end{equation*}
where $r_\alpha$ denotes the Matérn covariance given in \todo{\eqref{eq:matern_cov_2par}, $I = [0,50]$ is the domain we are considering,} and $\hat{r}_\alpha$ represents the approximation obtained using the methods under consideration.
 \begin{figure}[t]
  \begin{center}
	\includegraphics[height=0.57\linewidth]{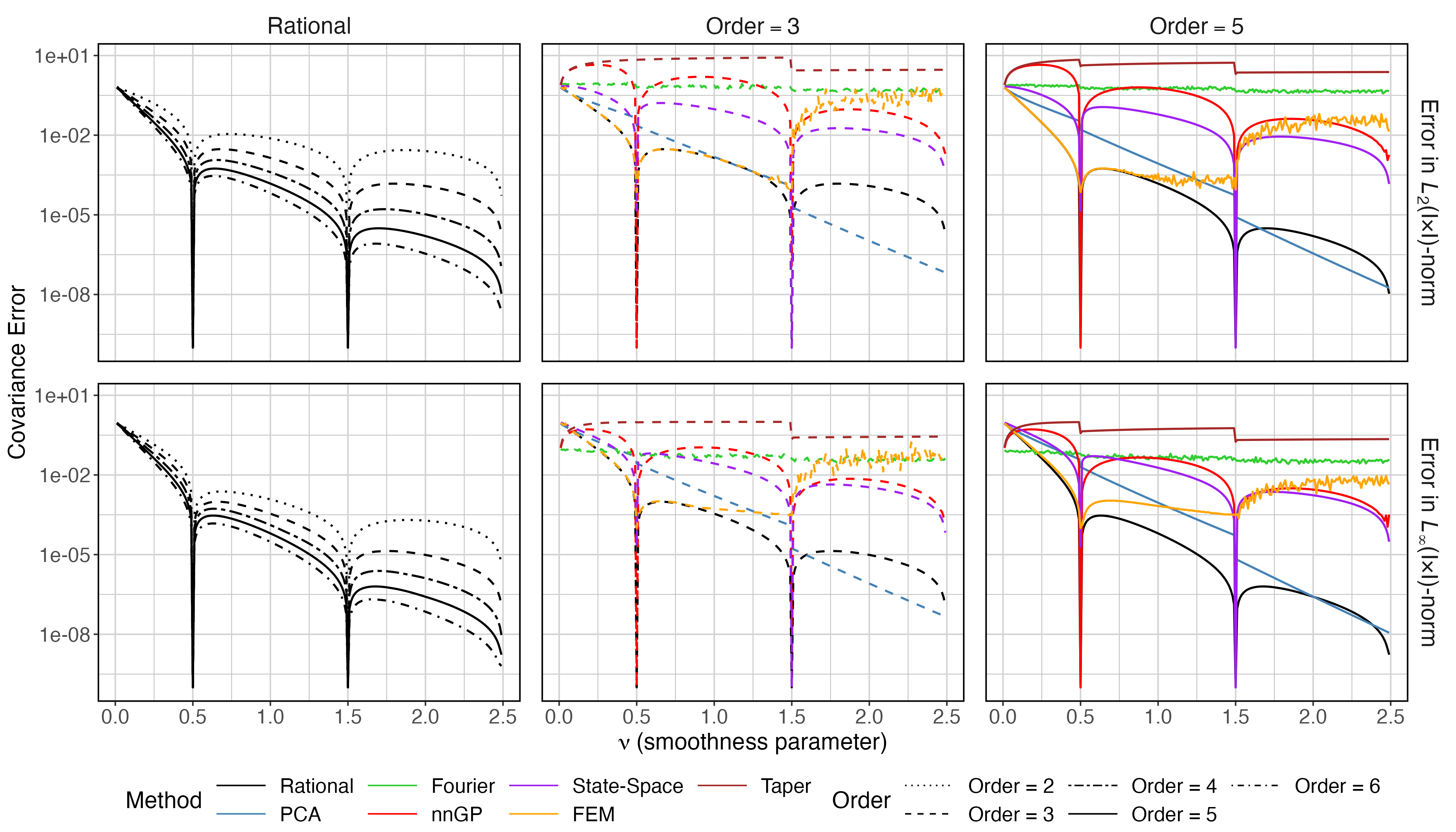}
  \end{center}
  \caption{Covariance errors for different values of $\nu$ and different orders of approximation.}
  \label{fig:covarianceerror}
\end{figure}
\todo{Figure~\ref{fig:covarianceerror} illustrates the resulting $L_2$ and supremum norm errors in the covariance approximation for the different methods when $n = 5000$.
The left panel illustrates the error for our method with varying choices of $m$. Notably, the covariance error decreases rapidly to zero as $m$ increases, with the rate of decrease becoming faster for larger values of $\nu$. Moreover, one can observe that, especially for $\nu$ away from zero, each increase in $m$ leads to improvements on the scale of orders of magnitude.}
In the remaining two panels, we show the error for $m=3$ and $m=5$ and compare these errors to the error from the competing methods, calibrated to have the same cost of prediction as described above. We can note that the competing methods are less accurate and in fact, one can observe that the nnGP method becomes slightly unstable for large values of smoothness parameter $\nu$, whereas the FEM method becomes very unstable due to the calibrated mesh being very fine. It is also noteworthy that the mesh is fine enough that when the method is stable (which mean $\nu < 0.9$ for $m=3$ and $\nu < 0.5$ for $m=5$), we can only see the rational approximation error, as both FEM and rational methods overlap. 
This shows that the FEM method is an excellent choice when stable, but that the rational is better because it is more stable and does not require choosing a mesh or a boundary extension. The PCA method is more accurate for large values of $\nu$, but one should recall that it is a theoretical lower bound for low rank methods, and not a practically competitive method, as it has an $\mathcal{O}(n^3)$ construction cost unless the eigenvectors are known explicitly. Any practically useful low rank method, such as the process convolution approach \citep{higdon2002space} or fixed rank kriging \citep{Cressie2008}, will have larger errors, as can clearly be seen when considering the errors of the random Fourier features method. \todo{Additionally, it is important to note that for the values $\nu = 0.5, 1.5, 2.5$ within the considered interval $(0,3)$, the proposed method is exact. Further, Figure~\ref{fig:covarianceerror} highlights an important observation regarding the theoretical bounds and numerical results: The numerical errors behave  similar to the (approximate) theoretical bounds in Figure \ref{fig:theoretical_error}, where the error is monotonic in $m$ but not monotonic in $\{\alpha\}$.
}

 \begin{figure}[t]
  \begin{center}
	\includegraphics[height=0.57\linewidth]{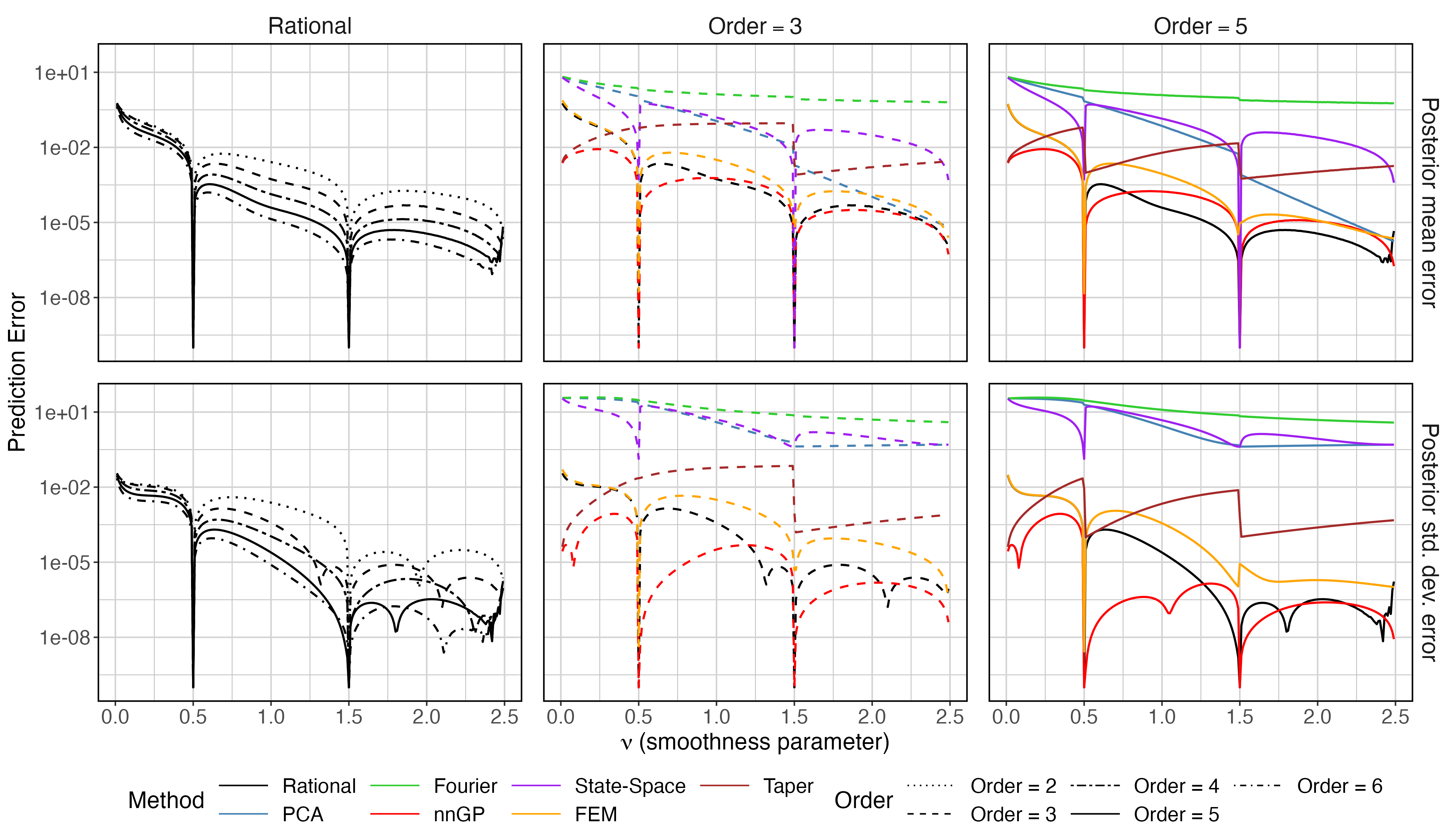}
  \end{center}
  \caption{\todo{Prediction errors, measured as mean squared errors of the posterior means and posterior standard deviations, evaluated on the interval $I = [0, 50]$ for $\sigma_e = 0.1$, $\rho = 2$, various values of $\nu$, and different orders of approximation.}}
  \label{fig:prediction}
\end{figure}

\subsection{Prediction errors}\label{sec:prediction_errors}
To compute the prediction error, we first compute the true posterior mean $\mv{\mu}_{\mv{u}|\mv{y}}$, the posterior mean $\hat{\mv{\mu}}_{\mv{u}|\mv{y}}$ under each approximate model and the corresponding $L_2$ errors \todo{$\|\mv{\mu}_{\mv{u}|\mv{y}} -\hat{\mv{\mu}}_{\mv{u}|\mv{y}}\|_{L_2}$, where 
$\|\mv{v}\|_{L_2}^2 =\frac{1}{n}\sum_{i=1}^{n}\mv{v}_i^2$ for $\mv{v}\in\mathbb{R}^n$.}
\todo{Similarly, we also compute the true posterior standard deviation $\mv{\sigma}_{\mv{u}|\mv{y}}$ and the posterior standard deviation $\hat{\mv{\sigma}}_{\mv{u}|\mv{y}}$ under each approximate model and the corresponding $L_2$ error
$\|\mv{\sigma}_{\mv{u}|\mv{y}} -\hat{\mv{\sigma}}_{\mv{u}|\mv{y}}\|_{L_2} $.}

\begin{figure}[t]
  \begin{center}
	\includegraphics[height=0.57\linewidth]{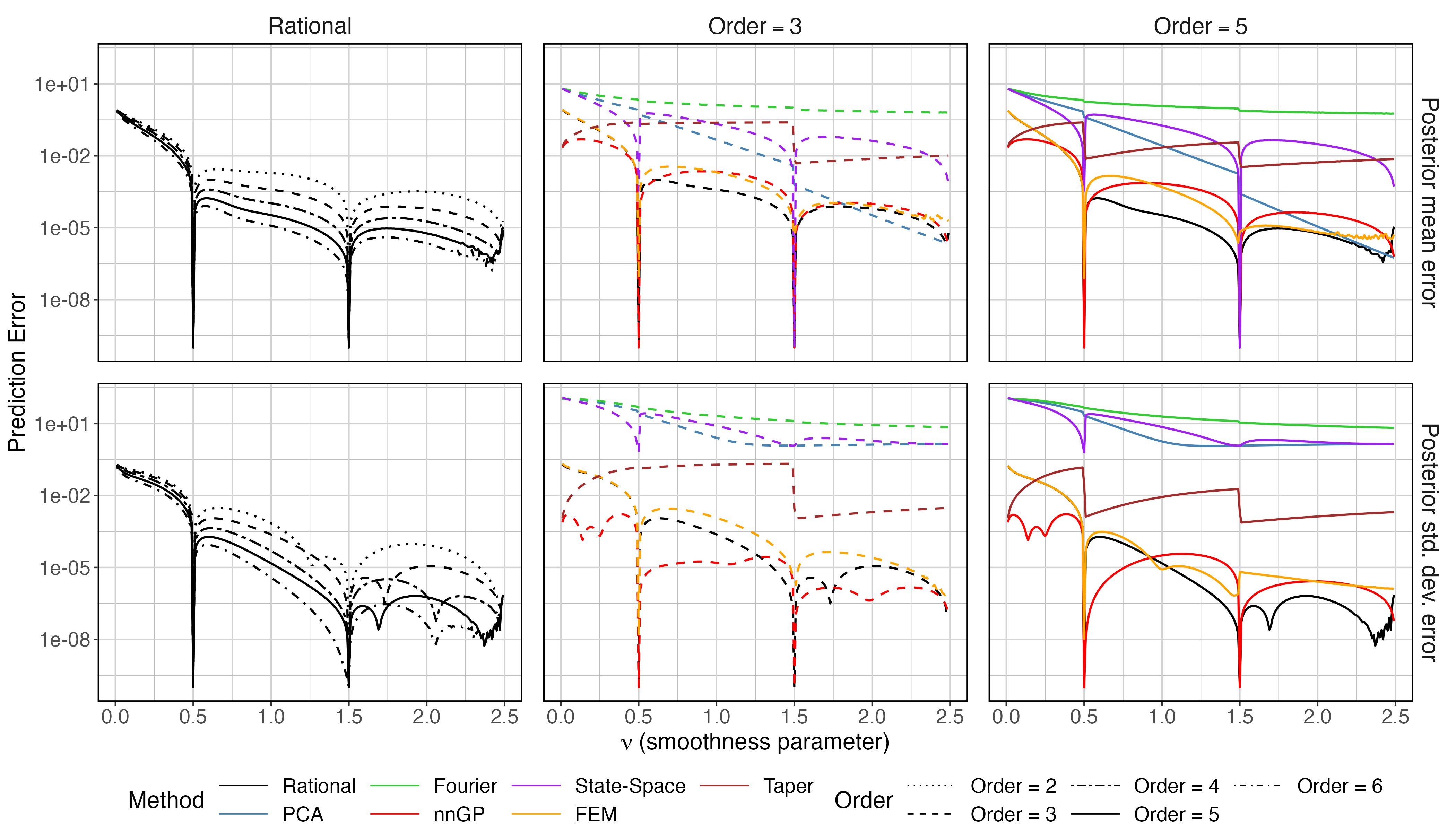}
  \end{center}
  \caption{\todo{Prediction errors, measured as mean squared errors of the posterior means and posterior standard deviations, evaluated on the interval $I = [0, 50]$ for $\sigma_e = \sqrt{0.1}$, $\rho = 2$, various values of $\nu$, and different orders of approximation.}}
  \label{fig:prediction_2}
\end{figure}

\todo{The prediction errors under the noise levels $\sigma_e = 0.1$ and $\sigma_e = \sqrt{0.1}$ are shown in Figure~\ref{fig:prediction} and Figure~\ref{fig:prediction_2}, respectively. The left columns of these figures show that the $L_2$ error of both the posterior means and posterior standard deviations decrease consistently and rapidly to zero for the rational approximation as $m$ increases.} 
\todo{The other two columns of the figures compare the errors of the different methods when calibrated to have the same cost as the rational approximation method with $m=3$ and $m=5$. Yhese results show that state-space method and the Fourier method are not competitive. The PCA method performs poorly for the posterior standard deviations, and also for the posterior mean unless $\nu$ is large. Additionally, using PCA at this cost requires prior knowledge of the eigenvectors, so it is typically not a feasible option. Further, the tapering method also performs poorly except for very small values of $\nu$ (essentially $\nu < 0.25$ where it is not calibrated). This leaves three methods which are competitive: the rational approximation, the FEM method and nnGP.} 

\todo{The accuracy of the FEM method is essentially equal to that of the rational approximation for $\nu \leq 0.5$, whereas it is slightly worse for $\nu > 0.5$. The nnGP method has the highest accuracy for $\nu < 1$ and $\sigma_{e} = 0.1$, but one should recall that the method is not calibrated for  $\nu<0.5$ so the comparison is not completely fair. For $\nu > 1$ and $\sigma_{e} = 0.1$, the two method have similar performance but the rational approximation is slightly better for the posterior mean. 
For $\sigma_e = \sqrt{0.1}$, the rational approximation outperforms nnGP and all other methods if $\nu > 1$, whereas for $\nu \in (0.5,1)$, the rational approximation has the best accuracy for the posterior mean but nnGP has the best accuracy for the standard deviation.}

Further comparisons for other choices of the number of prediction locations $n$, the number of  observation locations $n_{obs}$, and for different practical correlation ranges can be found in the accompanying Shiny application. The general conclusions for these choices align closely with those presented above. The main difference is that, for prediction tasks with a very small practical correlation range and an equal number of prediction and observation locations, nnGP tends to perform the best. This scenario is particularly favorable for nnGP, as all locations coincide with support points. Additionally, when $n > n_{obs}$, the rational approximation outperforms all methods, except for $\nu > 1.5$, where nnGP achieves the best performance in prediction tasks.

\subsection{Process approximation}
\todo{As observed in Section~\ref{sec:covariance_errors}, no alternative method comes close to the proposed method in terms of the stability and accuracy of the covariance approximation. However, in terms of prediction accuracy, the nnGP and FEM methods exhibit similar performance in certain scenarios. Specifically, for the FEM method, similar accuracies are observed for small values of $\nu$, such as $\nu < 1$ for order 3 and $\nu < 0.5$ for order 5. Also, for certain cases, the nnGP method  outperformed the proposed method in terms of prediction accuracy.}
However, one important difference of the nnGP method is that the approximation is strongly dependent on the support points, which typically is chosen as the observation locations in applications \citep{Datta2016}. In this section, we illustrate that this results in a poor approximation of the stochastic process, even though the approximation of the finite dimensional distribution at the support points is accurate. \todo{To this end, we consider the task of computing a joint confidence band for the latent process in two different scenarios.}
\todo{In both scenarios, we assume that we have data $\mathbf{y}=[y_1,y_2,\dotsc,y_n]$ generated as
$
y_i|u(\cdot)\sim \mathcal{N}(u(t_i),\sigma_e^2),
$
where $u$ is a centered Gaussian process with the Matérn covariance \eqref{eq:matern_cov}. As before, we consider the two noise levels $\sigma_e = 0.1$ and $\sigma_e = \sqrt{0.1}$. Based on these observations, we perform prediction at locations $p_1,\ldots,p_n$, and use the \texttt{excursions} package \citep{excursionspackage} to compute the upper and lower bounds, $c_l(s)$ and $c_u(s)$, of a joint confidence band so that 
$P(c_l(s) < u(s) < c_u(s) | \mathbf{y}, s \in \{p_1,\ldots, p_n\}) = 0.9.$}

\todo{In Scenario 1, we consider a forecasting setting where we construct a regular mesh of $1501$ points in the interval $[0,15]$. The first $1001$ points, evenly spaced in $[0,10]$, serve as observation locations. For prediction, we use 
all observation locations plus the next ${n \in \{0,1,\ldots,10,20,30,40,50,75,100,125,\ldots,500\}}$  consecutive points from the mesh as prediction locations.}
\todo{In Scenario 2, we instead use $250$ observation locations randomly sampled uniformly in the interval $[0,10]$, with a minimum spacing of $10^{-3}$ between locations ensured through resampling if needed, to ensure stability for nnGP. For for an increasing sequence of values $n$ between $0$ and $3000$, we consider $n$ evenly spaced prediction locations in the interval. The combined set of observation and prediction locations is processed to maintain the minimum spacing requirement while preserving all observation locations. To account for the randomness in the selection of observation locations, we repeat this scenario $10$ times and average the results.}

\todo{For each scenario, we compute the approximate coverage probabilities for nnGP, FEM, and the proposed method:
$
\tilde{p} = \tilde{P}(c_l(s) < u(s) < c_u(s) | \mathbf{y}, s \in \{p_1,\ldots, p_n\}),
$
where $\tilde{P}(\cdot|\cdot)$ denotes the posterior probability under the approximate model. If the process approximation is accurate, we expect that $\tilde{p} \approx 0.9$. Additionally, we evaluate the accuracy using the Kullback-Leibler (KL) divergence between the approximate and true posterior distributions.}

\begin{figure}[t]
  \begin{center}
	\includegraphics[width=\linewidth]{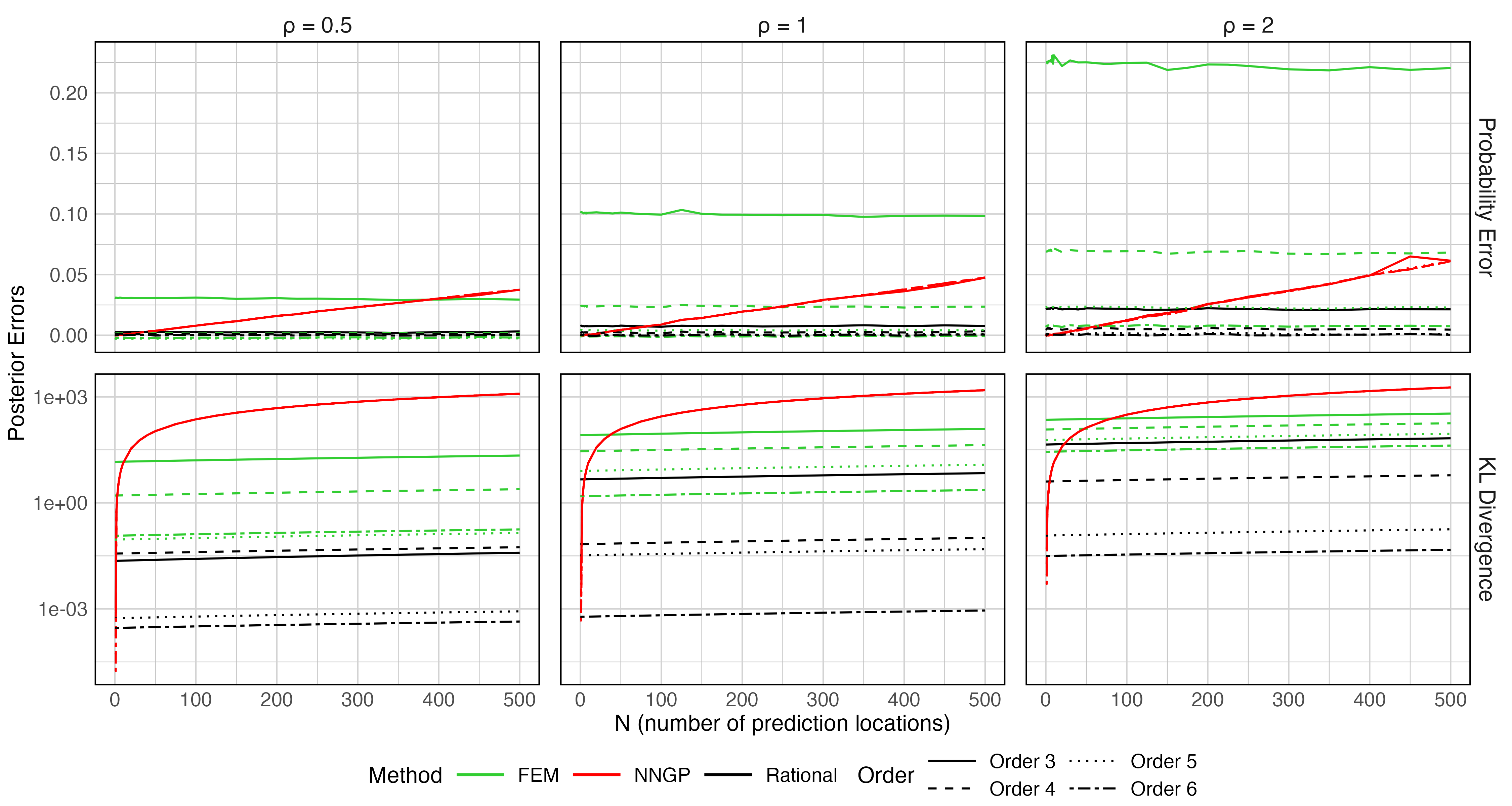}
  \end{center}
   \caption{\todo{Posterior probability errors and KL divergence under Scenario 1 with $\sigma_e = 0.1$ for  $\nu=1$, orders $m$ from $3$ to $6$, and three different practical correlation ranges $\rho$.}}
  \label{fig:post_prob_reg}
\end{figure}

\todo{We choose $\nu = 1$ and three values of the practical correlation range: $\rho = 0.5, 1,$ and $2$. For each scenario and each value of $n$, the order $m$ of the nnGP approximation is chosen so that the computational cost for evaluating the posterior mean at the prediction locations is the same as for the proposed method, and the support points for the nnGP approximation is kept fixed at the observation locations. For $n < 1000$, we use the calibration cost from $n = 1000$, as the observed cost in smaller cases primarily reflects computational overhead.}

\todo{The error $0.9 - \tilde{p}$ and KL divergences for the rational, nnGP and FEM methods are shown as functions of $n$ for Scenario 1 and noise level $\sigma_e = 0.1$ in Figure~\ref{fig:post_prob_reg} and for Scenario~2 and noise level $\sigma_e = \sqrt{0.1}$ in Figure~\ref{fig:post_prob2}. The results for Scenario 1 with $\sigma_e = \sqrt{0.1}$ and Scenario~2 with $\sigma_e = 0.1$ are very similar, so we do not include them here. However, the results for those cases can be seen in the accompanying Shiny app.}

\todo{As expected, the error of the nnGP method is small for $n=0$ but increases quickly with $n$ and has poor performance both in terms of probability approximation and KL divergence even for a very small number of prediction locations. The results for the proposed approximation are much more stable, which supports the claim that the proposed approximation is a better process approximation than the nnGP method. Further, the approximation is also better than the corresponding (calibrated) ones from the FEM method.}

\begin{figure}[t]
  \begin{center}
	\includegraphics[width=\linewidth]{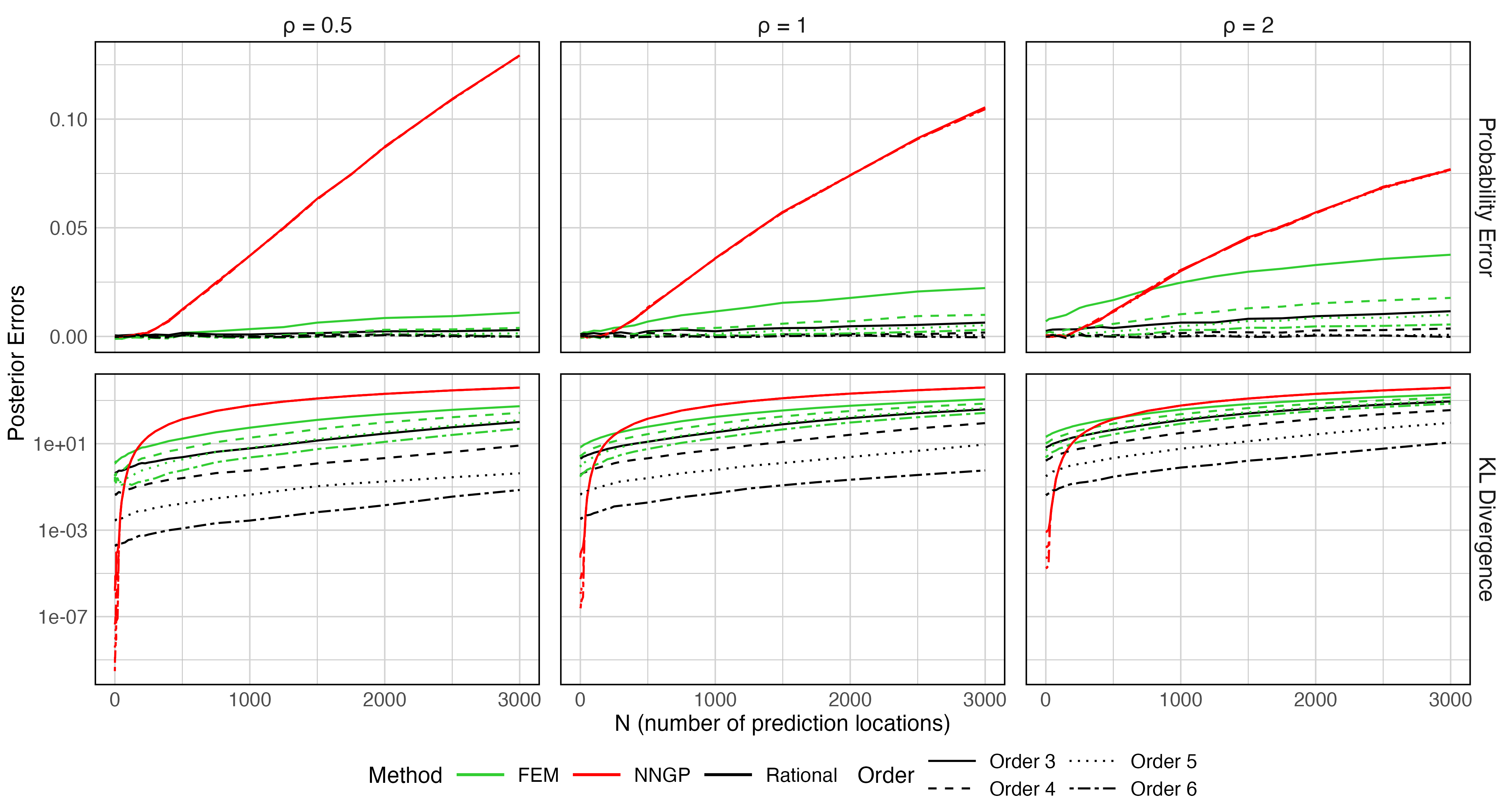}
  \end{center}
   \caption{\todo{Posterior probability errors and KL divergence under Scenario 2 with $\sigma_e = \sqrt{0.1}$ for $\nu=1$, orders $m$ from $3$ to $6$, and three different practical correlation ranges $\rho$.}}
  \label{fig:post_prob2}
\end{figure}

\todo{As a final study, we consider Scenario 1, in which we do forecasting, and compute the KL divergence between the true distribution of the field $(u(p_1),\ldots, u(p_n))$, where $u(\cdot)$ is a Gaussian process with the Matérn covariance function \ref{eq:matern_cov}, and their approximations by the rational, nnGP and FEM methods. We follow the approach of Scenario 1 and consider the first $1001$ points as support points for nnGP. We then compute the KL divergence between the true distribution of the field and the approximations by the different methods. The results are shown in Figure~\ref{fig:kl_div}, and we observe a similar situation as in the posterior distributions study.
The proposed method outperforms the FEM method consistently, and the nnGP method has good performance if all locations are support points. However, the nnGP error increases rapidly as the number of non-support points increases, and as soon as we have a few non-support points, the proposed method also outperforms nnGP.}

\todo{
  \begin{remark}
    We did not include Scenario 2 in the prior distribution study because the nnGP method is not stable for this scenario. The reason for this instability is that the nnGP method is highly dependent on the ordering of the support points and prediction points as shown in \cite[Section 4.2.1]{schafer2021sparse}. For the posterior calculations we were able to overcome this instability by using a Markov property for performing prediction for nnGP \citep{Datta2016}, which stabilized the posterior distributions. However, as seen in Figures \ref{fig:post_prob_reg} through \ref{fig:kl_div}, it is  unlikely that the conclusions would be different for  prior distributions under Scenario 2.
  \end{remark}
}

\begin{figure}[t]
  \begin{center}
	\includegraphics[height=0.35\linewidth]{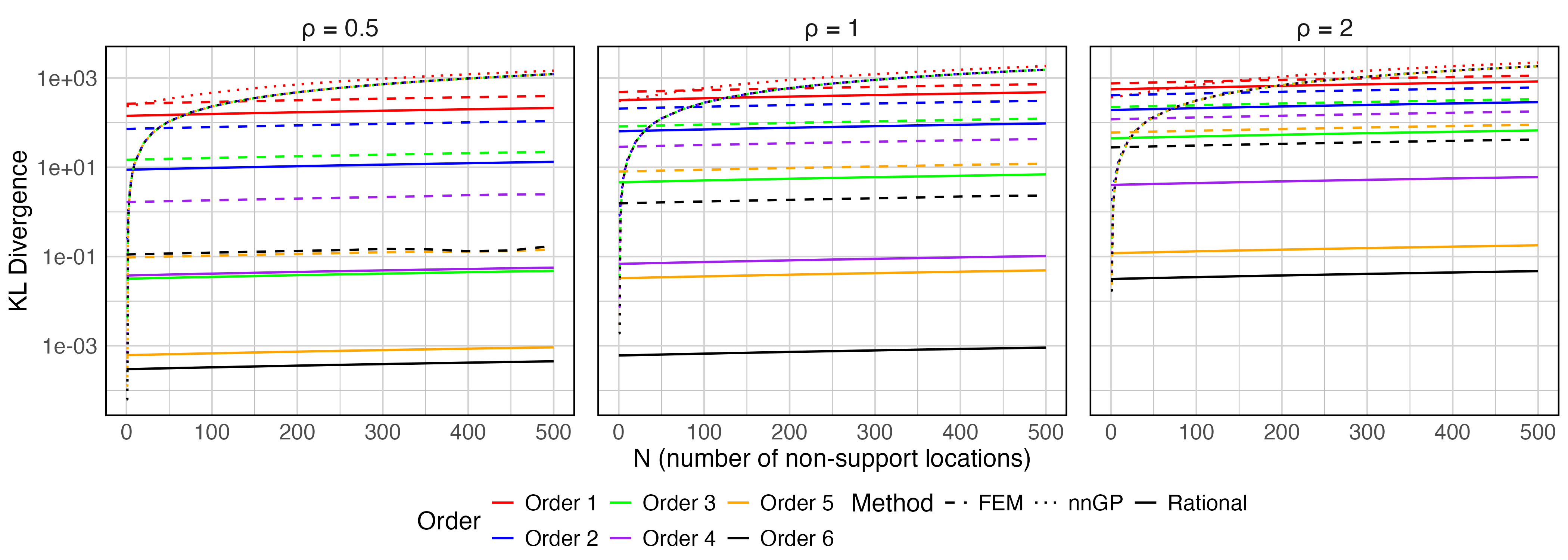}
  \end{center}
   \caption{\todo{KL divergences (in log scale) under Scenario 1 for $N$ non-support locations, $\nu=1$, orders $m$ from $1$ to $6$, and three different practical correlation ranges $\rho$.}}
  \label{fig:kl_div}
\end{figure}

\section{Discussion}\label{sec:discussion}  

We have introduced a broadly applicable method for working with Gaussian processes (GPs) with Matérn covariance functions on bounded intervals, facilitating statistical inference at linear computational cost. Unlike existing approaches, which often lack theoretical guarantees for accuracy, our method achieves exponentially fast convergence of the covariance approximation. This guarantees both computational efficiency and high accuracy. Furthermore, the method yields a Markov representation of the approximated process, which simplifies its integration into general-purpose software for statistical inference. 

Our experimental results demonstrate that the proposed method outperforms state-of-the-art alternatives, offering significant improvements in accuracy. The method has been implemented in the \texttt{rSPDE} package, which is compatible with popular tools such as \texttt{R-INLA} \citep{lindgren2015bayesian} and \texttt{inlabru} \citep{inlabru}, facilitating seamless integration into Bayesian hierarchical models. 

The two most competitive methods besides the proposed method were nnGP and FEM. The nnGP model provides accurate predictions, particularly for lower values of the smoothness parameter $\nu$, as illustrated in Figure~\ref{fig:prediction}. However, it falls short in terms of process approximation. Specifically, the covariance errors associated with nnGP are significantly larger than those from the rational approximation method, as shown in Figure~\ref{fig:covarianceerror}. Additionally, the probability approximations deteriorate markedly when observations are not contained in the support points, evidenced by \todo{Figures~\ref{fig:post_prob_reg} through \ref{fig:kl_div}}. This limitation suggests that while nnGP may provide good results in contexts focused on prediction, such as those commonly found in machine learning, it is less competitive in 
broader statistical applications which require accurate process approximations. The main advantage of the proposed method compared to the FEM method is a generally higher accuracy and stability \todo{(in particular for large values of $\nu$)}, and the fact that no mesh needs to be chosen.

\todo{To ensure a fair comparison, we assumed all parameters to be known throughout the simulation studies, providing a clear baseline using the true covariance function. We expect similar results also in cases with unknown parameters, such as when performing likelihood-based statistical inference, as iterative numerical procedures rely on intermediate approximations of the true covariance. However, this remains to be investigated.}

An exciting direction for future work involves extending our method to more general domains. As mentioned in Section~\ref{sec:rationalapprox}, the method can also be derived through a rational approximation of the covariance operator of the process. More precisely, consider a Gaussian process $u$ on a domain $\mathcal{D}$ defined as the solution to  
\begin{equation}\label{eq:spde_alt}
    L^{\beta} (\tau u) = \mathcal{W}, \quad \text{on } \mathcal{D},
\end{equation}  
where $\mathcal{W}$ is Gaussian white noise on $\mathcal{D}$, $\beta > 0$ is a real number, and $L$ is a densely defined self-adjoint operator on $L_2(\mathcal{D})$, so that the covariance operator of $u$ is given by $L^{-2\beta}$. For example, if we have $L = (\kappa^2 - \Delta)$ on $\mathbb{R}$, we obtain a Gaussian process with a stationary Mat\'ern covariance function \citep{whittle63}. Following the approach in \cite{bolin2024covariance}, one can approximate the covariance operator as  
\begin{equation}\label{eq:rat_approx_cov}
    L^{-2\beta} \approx L^{-\lfloor 2\beta\rfloor} \left( \sum_{i=1}^m c_i (L - p_i I)^{-1} + k I \right),
\end{equation}  
where $\{c_i\}_{i=1}^m$, $\{p_i\}_{i=1}^m$ and $k$ are real numbers, satisfying $p_i \leq 0$, $c_i > 0$, and $k > 0$ 
by Proposition~\ref{prp:partial_fractions},
and $I: L_2(\mathcal{D}) \to L_2(\mathcal{D})$ is the identity operator. If suitable conditions on $L$ hold, this approximation expresses $L^{-2\beta}$ as a linear combination of covariance operators. We will refer to this approach as the covariance-based approach.

By computing the finite-dimensional distributions corresponding to the operators $L^{-\lfloor 2\beta\rfloor}$ and $L^{-\lfloor 2\beta\rfloor}(L - p_i I)^{-1}$ for $i = 1, \dots, m$, we can approximate the finite-dimensional distributions of $L^{-2\beta}$ and thus the solution to \eqref{eq:spde_alt}. Assume further that these covariance operators are induced by covariance functions. Let $\varrho_\beta(\cdot, \cdot)$ denote the covariance function associated with $L^{-2\beta}$, $\widehat{\varrho}_{\beta}(\cdot, \cdot)$ the covariance function corresponding to $k L^{-\lfloor 2\beta\rfloor}$, and $\widehat{\varrho}_{\beta, i}(\cdot, \cdot)$ the covariance function associated with $c_i L^{-\lfloor 2\beta\rfloor}(L - p_i I)^{-1}$ for $i = 1, \dots, m$. Then, \eqref{eq:rat_approx_cov} implies that
\begin{equation}\label{eq:cov_function_rat_approx}
    \varrho_\beta(x, y) \approx \widehat{\varrho}_{\beta}(x, y) + \sum_{i=1}^m \widehat{\varrho}_{\beta, i}(x, y), \quad x, y \in \mathcal{D},
\end{equation}
which corresponds exactly to the expression in Proposition~\ref{cov_prop}, demonstrating the equivalence between the spectral and covariance-based approaches.

It is important to note, however, that both approaches—the covariance-based and the spectral approach—yield approximations that do not converge to the true covariance operator when $\mathcal{D} = \mathbb{R}^d$ for $d \geq 1$. Nonetheless, one can instead consider compact domains, $\mathcal{D}$, where convergence of the approximation can generally be established. 

This formulation offers a pathway for extending our method to more complex (compact) domains. For instance, starting with $L = (\kappa^2 - \Delta)$ on a circle, one can define rational approximations for periodic Matérn fields. More generally, the approach can be adapted to network-like domains, following recent developments in Gaussian Whittle–Matérn fields on metric graphs \citep{bolin2023statistical, bolin2024gaussian}. These extensions, including their theoretical and computational aspects, will be explored in future work. Finally, if $\mathcal{D}$ is a $d$-dimensional manifold (which also includes compact subsets of $\mathbb{R}^d$), where $d\geq 2$, this approach is not expected to provide sparse approximations unless $\mathcal{D}$ is a Cartesian product of intervals, and $L$ is a separable differential operator, that is, $L = \sum_{i=1}^d L_i$, where $L_i$ is a differential operator acting only on $x_i$, where $x=(x_1,\ldots,x_d)\in \mathcal{D}$. This implies that, while the approximation may still be applicable for manifolds of dimension greater than one, it is generally not practical unless the model is separable. Such separable models, which have been explored in previous work (see, e.g., \cite{chen2022kernel}), are rarely suitable for modeling spatial data. 

\todo{Another advantage with the covariance-based formulation is that is that it can be combined with FEM approximations, and because by taking advantage of the fact that \eqref{eq:cov_function_rat_approx} is stationary, one could use this to avoid having to rely on boundary extensions to remove boundary effects for the FEM approach. As this is a common problem with the FEM approach, which increases the computational cost, this is an interesting topic to investigate in future work. }

\appendix

\section{LDL Algorithm and proof of Proposition~\ref{ra_chol}}\label{ldlsection}
In this section, we explain how to construct the matrices $\mv{L}_i$ and $\mv{D}_i$ of Proposition~\ref{ra_chol}. To simplify the notation, let $p$ denote the size of the vector $\mv{u}_{i,j}$ (which is either $\max(\lfloor\alpha\rfloor,1)$ if $i=0$ or $\lceil\alpha\rceil$ if $i>0$).
Introduce the matrices $\mv{B}^1 = \mv{L}_{1,1}$ and $\mv{B}^k = [\mv{L}_{k,k-1}\, \mv{L}_{k,k}], k>1,$ corresponding to the part of $\mv{L}_i$ that is related to the $k$th location. Similarly, let  $\mv{F}^k$ denote the diagonal matrix corresponding to the part of $\mv{D}^{-1}$ that is related to the $k$th location. That is, $\mv{D} = \text{diag}(\mv{F}^1,\ldots,\mv{F}^n)^{-1}.$
These are low-dimensional matrices with a dimension $2p\times p$ and $p\times p$, respectively.
Further, introduce the covariance matrices $\mv{\Sigma}^1 = \mv{r}_i(t_1,t_1)$ and 
$$
\mv{\Sigma}^k = \begin{bmatrix}
    \mv{r}_i(t_1,t_1) & \mv{r}_i(t_1,t_2) \\
    \mv{r}_i(t_2,t_1) & \mv{r}_i(t_2,t_2)
\end{bmatrix}, \quad k=2,\ldots,n,
$$
and for an $n\times n$ matrix $\mv{M}$, let $\mv{M}_{a:b,c:d}$, for natural numbers $a\leq b \leq n$ and $c\leq d \leq n$ denote the submatrix obtained by extracting rows $a,\ldots,b$ and columns $c,\ldots, d$ from $\mv{M}$. 

We are now ready to construct the elements in the matrices. We set $\mv{F}^1_{1,1} = 1/\mv{\Sigma}^1_{1,1}$ and 
\begin{align*}
\mv{F}^1_{j,j} &= \mv{\Sigma}^1_{j,j} -  \mv{\Sigma}^1_{j,1:(j-1)}(\mv{\Sigma}^1_{1:(j-1),1:(j-1)})^{-1}\mv{\Sigma}^1_{1:(j-1),j}, & &j=2,\ldots,p \\
\mv{F}^k_{j,j} &= \mv{\Sigma}^k_{p+j,p+j} -  \mv{\Sigma}^k_{p+j,1:(p+j-1)}(\mv{\Sigma}^k_{1:(p+j-1),1:(p+j-1)})^{-1}\mv{\Sigma}^k_{1:(p+j-1),j}, & &j=1,\ldots,p, k>1.
\end{align*}
Further, we set $\mv{B}^1_{j,j} = 1$ and $\mv{B}^1_{j,j+i} = 0$ for $i>0$ and 
$$
\mv{B}^1_{j,1:(j-1)} = - \mv{\Sigma}^1_{j,1:(j-1)}(\mv{\Sigma}^1_{1:(j-1),1:(j-1)})^{-1},\quad i=2,\ldots,p.
$$
Finally, we set $\mv{B}^k_{j,p+j} = 1$ and $\mv{B}^k_{j,p+j+i} = 0$ for $i>0$ and
$$
\mv{B}^k_{j,1:(j-1)} = - \mv{\Sigma}^k_{j,1:(p+j-1)}(\mv{\Sigma}^1_{1:(p+j-1),1:(p+j-1)})^{-1},\quad i=1,\ldots,p, k>1.
$$
These expressions follow directly from using the fact that $\mv{u}_i$ is a first order Markov process and then writing the joint density of $\mv{u}_i$ as 
$$
\pi(\mv{u}_i) = \pi(\mv{u}_{i,1})\prod_{k=w}^n \pi(\mv{u}_{i,k}|\mv{u}_{i,k-1}).
$$
After this, standard results for conditional Gaussian densities give the expressions above \citep[see, for example][]{Datta2016}. To compute the elements in each block, we need to perform $p$ solves of matrices of size $p, p+1, \ldots, 2p$, the total cost of this is $\mathcal{O}(p^4)$ and since we have $n$ blocks, the total cost is thus $\mathcal{O}(np^4)$.


\section{Collected proofs}\label{markovsection}

We start with a simple technical lemma needed for the proof of Theorem \ref{ra_bound}.

\begin{lemma}\label{lem:int_ineq}
    Let $h:\mathbb{R}\to\mathbb{R}$ be a measurable function. Then, the following inequality holds:
    $$\int_a^b \int_a^b h(x-y)^2 dx\, dy \leq (b-a) \int_{a-b}^{b-a} h(x)^2 dx,$$
    where $a,b\in\mathbb{R}, a<b$. 
\end{lemma}
\begin{proof}
    We have, by the change of variables $u = x-y$, that
    \begin{align*}
        \int_a^b\int_a^b h(x-y)^2 dx\, dy&= \int_a^b \int_{a-y}^{b-y} h(u)^2 du\,dy\\
        &\leq \int_a^b \int_{a-b}^{b-a} h(u)^2 du\,dy = (b-a)\int_{a-b}^{b-a} h(u)^2 du,
    \end{align*}
where we used the fact that $a\leq y\leq b$.
\end{proof}

\begin{proof}[of Theorem \ref{ra_bound}] 
Our idea is to use the exponential convergence of the best rational approximation with respect to the $L_\infty$ norm. To such an end, \citet[][Theorem 1]{stahl2003best}, gives us that for every $\alpha \in (0,1)$, there exist polynomials of degree $m$, $p_m(\cdot)$ and $q_m(\cdot)$ such that the following exponential bound holds:
	\begin{equation}\label{eq:expconvbound}
	\sup_{x\in [0,1]}\left|x^{\alpha} - \frac{p_m(x)}{q_m(x)}\right| \leq C_\alpha e^{-2\pi \sqrt{\alpha m}},
	\end{equation}
    where \todo{$C_\alpha>0$} is a constant that only depends on $\alpha$. 
Now, let $g(w) = (1+\kappa^{-2}w^2)^{-1}$ and observe that for every $w\in\mathbb{R}$, we have $g(w) \in [0,1]$. Further, let $\alpha = \nu+0.5$ and $f_{\alpha}(\cdot)$ be the spectral density of a stationary Gaussian process with a Matérn covariance function \eqref{eq:matern_cov}. Then,  $f_{\alpha}(w) = A\sigma^2\kappa^{-2\alpha}(g(w))^{\alpha}$, \todo{where $A=\sqrt{2}\kappa^{2\nu}\Gamma(\nu+\nicefrac{1}{2})\Gamma(\nu)^{-1}$}. 	\todo{Observe that we have the following decompositions:}
  $$\todo{f_{\alpha}(w) = f_{\lfloor\alpha\rfloor}(w) (g(w))^{\{\alpha\}}\quad\hbox{and}\quad f_{m,\alpha}(w) = f_{\lfloor\alpha\rfloor}(w) \frac{p_m(g(w))}{q_m(g(w))}.}$$
  We also have, for any $\alpha>1/2$, that 
  \begin{align*}
    \todo{\sup_{w \in \mathbb{R}}\left| (g(w))^{\{\alpha\}} - \frac{p_m(g(w))}{q_m(g(w))}\right| = \sup_{x \in [0,1]} \left| x^{\{\alpha\}} - \frac{p_m(x)}{q_m(x)}\right| 
    \leq C_{\{\alpha\}} e^{-2\pi \sqrt{\{\alpha\} m}},}
    \end{align*}
 where we used \eqref{eq:expconvbound} and that $0\leq x\leq 1$. \todo{Because $\sup_{w\in\mathbb{R}} |f_{\lfloor\alpha\rfloor}(w)| = |f_{\lfloor\alpha\rfloor}(0)| = A\sigma^2\kappa^{-2\alpha}$,}
 \begin{align}
  \todo{\sup_{w \in \mathbb{R}}\left| f_{\alpha}(w) - f_{m,\alpha}(w)\right|} &\todo{\leq \sup_{w \in \mathbb{R}}|f_{\lfloor\alpha\rfloor}(w)| \sup_{w\in\mathbb{R}}\left|(g(w))^{\{\alpha\}} - \frac{p_m(g(w))}{q_m(g(w))}\right|}\nonumber\\
  &\todo{\leq A\sigma^2\kappa^{-2\alpha}C_{\{\alpha\}} e^{-2\pi \sqrt{\{\alpha\} m}}}.\label{eq:bound_spec_dens_f}
 \end{align}
Now, let $I=[a,b]\subset\mathbb{R}$ be an interval in $\mathbb{R}$ and $1_{I} : \mathbb{R} \rightarrow \mathbb{R}$ denote the indicator function such that $1_{I}(x) = 1$ if $x\in I$ and $1_{I}(x) = 0$ otherwise. We then have by Plancherel's theorem \cite[e.g.][Corollary 3.13]{mclean}, the convolution theorem \cite[e.g.][p.73]{mclean}, Lemma~\ref{lem:int_ineq}, \todo{and the bound \eqref{eq:bound_spec_dens_f} above, that}
	\begin{align*}
		\|r_{m,\alpha} - r_\alpha\|_{L_2(I\times I)}^2 &\leq (b-a) \|\varrho_{m,\alpha} - \varrho_\alpha\|_{L_2(a-b,b-a)}^2 = (b-a) \|(\varrho_{m,\alpha} - \varrho_\alpha)1_{[a-b,b-a)]}\|_{L_2(\mathbb{R})}^2\\
		&= (b-a) \left\|(f_{m,\alpha} - f_\alpha) \ast 2(b-a)\frac{\sin((b-a)w)}{(b-a)w}\right\|_{L_2(\mathbb{R})}^2 \\
		&\todo{\leq (b-a) \sup_{w\in \mathbb{R}}\left((g(w))^{\{\alpha\}} - \frac{p_m(w)}{q_m(w)}\right)^2 \left\|2(b-a)\frac{\sin((b-a)w)}{(b-a)w}\right\|_{L_2(\mathbb{R})}^2} \\
  &\todo{\leq 4(b-a)^3  A^2\sigma^4\kappa^{-4\alpha}C_{\{\alpha\}}^2 e^{-4\pi \sqrt{\{\alpha\} m}} \left\|\frac{\sin((b-a)w)}{(b-a)w}\right\|_{L_2(\mathbb{R})}^2}\\
  & = \todo{\pi (b-a)^2 A^2\sigma^4\kappa^{-4\alpha}C_{\{\alpha\}}^2 e^{-4\pi \sqrt{\{\alpha\} m}},}
	\end{align*}
	\todo{where $\ast$ denotes the convolution and we used the fact that $\left\|\frac{\sin((b-a)w)}{(b-a)w}\right\|_{L_2(\mathbb{R})}^2 = \pi/(b-a)$}. 
  \todo{Let us now consider another strategy that helps us obtain another bound for $\alpha>1$, that gets better as $\alpha$ increases. First, observe that by Young's convolution inequality \cite[Theorem 3.9.4]{bogachev2007measure}, we have}
  $\todo{\|f\ast g\|_{L_2(\mathbb{R})}\leq \|f\|_{L_1(\mathbb{R})}\|g\|_{L_2(\mathbb{R})},}$
  \todo{which yields,}
  \begin{align*}
    \todo{\left\|f_{\lfloor\alpha\rfloor}(w) \ast \frac{\sin((b-a)w)}{(b-a)w}\right\|_{L_2(\mathbb{R})}} &\todo{\leq \|f_{\lfloor\alpha\rfloor}(w)\|_{L_1(\mathbb{R})} \left\|\frac{\sin((b-a)w)}{(b-a)w}\right\|_{L_2(\mathbb{R})}}\\
    &\todo{= A\sigma^2\kappa^{-2\alpha}M_{\lfloor\alpha\rfloor,\kappa} \sqrt{\frac{\pi}{b-a}}},
  \end{align*}
  \todo{where $M_{0,\kappa} = \infty$ and $M_{n,\kappa} =  \|(g(w))^{n}\|_{L_1(\mathbb{R})} = \kappa \pi \Gamma(2n-1)/(4^{n-1}\Gamma(n)^2), n\geq 1$. Therefore, by proceeding as in the first case, and using the above inequality, we have}
  \begin{align*}
		\todo{\|r_{m,\alpha} - r_\alpha\|_{L_2(I\times I)}^2} &\leq \todo{(b-a) \left\|(f_{m,\alpha} - f_\alpha) \ast 2(b-a)\frac{\sin((b-a)w)}{(b-a)w}\right\|_{L_2(\mathbb{R})}^2} \\
		&\todo{\leq 4(b-a)^3 \sup_{w\in \mathbb{R}}\left((g(w))^{\{\alpha\}} - \frac{p_m(w)}{q_m(w)}\right)^2 \left\|f_{\lfloor\alpha\rfloor}(w) \ast\frac{\sin((b-a)w)}{(b-a)w}\right\|_{L_2(\mathbb{R})}^2}\\
  &\todo{\leq 4\pi (b-a)^2  A^2\sigma^4\kappa^{-4\alpha} M_{\lfloor\alpha\rfloor,\kappa}^2 C_{\{\alpha\}}^2 e^{-4\pi \sqrt{\{\alpha\} m}}}.\\
	\end{align*}
  \todo{Finally, by combining both inequalities above, we obtain }
  $$\todo{\|r_{m,\alpha} - r_\alpha\|_{L_2(I\times I)} \leq 2 \sqrt{\pi} (b-a) A \sigma^2 \kappa^{-2\alpha} \min\{1, M_{\lfloor\alpha\rfloor,\kappa}\} C_{\{\alpha\}} e^{-2\pi \sqrt{\{\alpha\} m}}}.$$
This proves \eqref{total_error_bound}. 

Now, for the $L_\infty(I \times I)$ estimate. Recall the notation from the first part of the proof. Also, recall the definition of the Sobolev space $H^s(\mathbb{R})$ for $s\geq 0$ \cite[see, e.g.][p.75-76]{mclean}. Then, for $s>\nicefrac{1}{2}$, we have by definition of the Sobolev norm that
\begin{align*}
\|\varrho_\alpha - \varrho_{m,\alpha}\|_{H^s(\mathbb{R})}&= \|(1+w^2)^{\nicefrac{s}{2}}(f_\alpha - f_{m,\alpha})\|_{L_2(\mathbb{R})}\\
&\todo{\leq  \sup_{w\in \mathbb{R}}\left|(g(w))^{\{\alpha\}} - \frac{p_m(w)}{q_m(w)}\right| \|f_{\lfloor \alpha\rfloor} (1+w^2)^{\nicefrac{s}{2}}\|_{L_2(\mathbb{R})}}\\
&\todo{\leq A \sigma^2\kappa^{-2\alpha}C_{\{\alpha\}}e^{-2\pi \sqrt{\{\alpha\}m}}\|(1+w^2)^{s/2-\lfloor \alpha \rfloor}\|_{L_2(\mathbb{R})}},
\end{align*}
Now, because $\alpha>1$, we have that $\lfloor \alpha\rfloor \geq 1$. Further,  $\|(1+w^2)^{s/2-\lfloor \alpha \rfloor}\|_{L_2(\mathbb{R})} <\infty$ if, and only if, $2\lfloor \alpha \rfloor-s>\nicefrac{1}{2}$. In particular, since $\lfloor \alpha\rfloor \geq 1$ we can find $\todo{s_\alpha}>\nicefrac{1}{2}$ such that ${\|(1+w^2)^{\todo{s_\alpha}/2-\lfloor \alpha \rfloor}\|_{L_2(\mathbb{R})} <\infty}$. Therefore, if $\alpha>1$, we can find $\todo{s_\alpha}>\nicefrac{1}{2}$ such that
$$\|\varrho_\alpha - \varrho_{m,\alpha}\|_{H^{\todo{s_\alpha}}(\mathbb{R})} \leq \todo{A \sigma^2\kappa^{-2\alpha}\breve{C}_\alpha C_{\{\alpha\}}e^{-2\pi \sqrt{\{\alpha\}m}}},$$
where the constant $\todo{\breve{C}_\alpha}$ depends on $\todo{s_\alpha}$, \todo{which only depends on $\alpha$}. Now, from Sobolev embedding, see, e.g., \cite[Theorem~3.26]{mclean}, \todo{since $s_\alpha>\nicefrac{1}{2}$}, we have that there exists $\todo{C_{sob,\alpha}}>0$, depending only on \todo{$s_\alpha$,  thus only on $\alpha$}, such that for every $x\in\mathbb{R}$,
$\todo{|\varrho_\alpha(x) - \varrho_{m,\alpha}(x)| \leq C_{sob,\alpha} \|\varrho_\alpha - \varrho_{m,\alpha}\|_{H^{s_\alpha}(\mathbb{R})}},$
so that
$$\todo{\sup_{x\in\mathbb{R}} |\varrho_\alpha(x) - \varrho_{m,\alpha}(x)| \leq C_{sob,\alpha} \|\varrho_\alpha - \varrho_{m,\alpha}\|_{H^s(\mathbb{R})} \leq A \sigma^2\kappa^{-2\alpha}C_{sob,\alpha}\breve{C}_\alpha C_{\{\alpha\}}e^{-2\pi \sqrt{\{\alpha\}m}}.}$$
Therefore,
$$\todo{\|r_{m,\alpha} - r_\alpha\|_{L_\infty(I\times I)} \leq \sup_{x\in\mathbb{R}} |\varrho_\alpha(x) - \varrho_{m,\alpha}(x)| \leq A \sigma^2\kappa^{-2\alpha}K_\alpha e^{-2\pi \sqrt{\{\alpha\}m}},}$$
\todo{where $K_\alpha>0$ is a constant that only depends on $\alpha$.} This proves \eqref{sup_bound}. 
\end{proof}

\todo{We now move to the proof of Proposition \ref{prp:partial_fractions}. To this end, we start by recalling the following result \citep[Lemma 2.1]{saff1995asymptotic}:}

\begin{lemma}\label{lem:roots_rat_approx}
  \todo{Let $0 < \alpha < 1$.}
  \begin{enumerate}
      \item \todo{The best uniform rational approximation $R_m^*(x)$ of the function $x^{\alpha}$ on $[0,1]$, among rational functions whose numerator has degree at most $m$, is such that both the numerator and denominator have degree exactly $m$.
      \item All $m$ zeros $\tilde{z}_1, \dots, \tilde{z}_m$ and poles $\tilde{p}_1, \dots, \tilde{p}_m$ of $R_m^*(\cdot)$ lie on the negative half-axis and are interlacing; i.e., with an appropriate numbering, 
      \begin{equation}\label{eq:zeros_poles_signs} 
      0 > \tilde{z}_1 > \tilde{p}_1 > \tilde{z}_2 > \tilde{p}_2 > \dots > \tilde{z}_m > \tilde{p}_m > -\infty.
      \end{equation}
      }
  \end{enumerate}
\end{lemma}

\todo{We are now in a position to prove Proposition \ref{prp:partial_fractions}.}

\begin{proof}[of Proposition \ref{prp:partial_fractions}]
  \todo{Define $\widetilde{P}_m(x) = \sum_{i=0}^m a_i x^i$ and $\widetilde{Q}_m(x) = \sum_{i=0}^m b_i x^i$ so that the best rational approximation of $x^{\alpha}$ on $[0,1]$ is $R_m^*(x) = \widetilde{P}_m(x)/\widetilde{Q}_m(x)$. Then, by Lemma~\ref{lem:roots_rat_approx}, 
  $\widetilde{P}_m(x) = m_{\widetilde{P}} (x- \tilde{z}_1) (x-\tilde{z}_2) \cdots (x-\tilde{z}_m),$
  and
  $\widetilde{Q}_m(x) = m_{\widetilde{Q}} (x- \tilde{p}_1) (x-\tilde{p}_2) \cdots (x-\tilde{p}_m),$
  where $m_{\widetilde{P}}$ and $m_{\widetilde{Q}}$ are constants and $\{\tilde{z}_i\}_{i=1}^m, \{\tilde{p}_i\}_{i=1}^m$ satisfy \eqref{eq:zeros_poles_signs}.
  Let $\sgn(x) = |x|/x$ denote the sign of $x\neq 0$. Since $\tilde{z}_i,\tilde{p}_i <0$ for $i=1,\ldots,m$, we have that for every $x\in (0,1)$, $\sgn(\widetilde{P}_m(x)) = \sgn(m_{\widetilde{P}})$ and $\sgn(\widetilde{Q}_m(x)) = \sgn(m_{\widetilde{Q}})$. Since $x^{\{\alpha\}} > 0$ for $x\in (0,1)$, $\widetilde{P}_m(x)$ and $\widetilde{Q}_m(x)$ must have the same sign for $x\in (0,1)$, so that 
  \begin{equation}\label{eq:signs_mP_mQ}
    \sgn(m_{\widetilde{P}}) = \sgn(m_{\widetilde{Q}}).
  \end{equation}
  Now, observe that 
  $P_m(x) = x^m \widetilde{P}_m(x^{-1})$ and $Q_m(x) = x^m \widetilde{Q}_m(x^{-1}).$
  Thus,
  \begin{align}
    P_m(x) &= x^m \widetilde{P}_m(x^{-1}) = m_{\widetilde{P}} (1-x\tilde{z}_1)\cdots (1-x\tilde{z}_m)\nonumber\\
    &= m_{\widetilde{P}} \tilde{z}_1 \cdots \tilde{z}_m (-1)^m \left(x-\frac{1}{\tilde{z}_1}\right)\cdots \left(x-\frac{1}{\tilde{z}_m}\right) \nonumber\\
    &= m_P (x-z_1) \cdots (x-z_m),\label{eq:P_m}
  \end{align}
  where $m_P = m_{\widetilde{P}} \tilde{z}_1 \cdots \tilde{z}_m (-1)^m$ and $z_i = 1/\tilde{z}_i$ for $i=1,\ldots,m$. Similarly, we have that
  \begin{align}
    Q_m(x) &= x^m \widetilde{Q}_m(x^{-1}) = m_{\widetilde{Q}} (1-x\tilde{p}_1)\cdots (1-x\tilde{p}_m)\nonumber\\
    &= m_{\widetilde{Q}} \tilde{p}_1 \cdots \tilde{p}_m (-1)^m \left(x-\frac{1}{\tilde{p}_1}\right)\cdots \left(x-\frac{1}{\tilde{p}_m}\right) \nonumber\\
    &= m_Q (x-p_1) \cdots (x-p_m),\label{eq:Q_m}
  \end{align}
  where $m_Q = m_{\widetilde{Q}} \tilde{p}_1 \cdots \tilde{p}_m (-1)^m$ and $p_i = 1/\tilde{p}_i$ for $i=1,\ldots,m$. Now, observe that
  $$\sgn(m_P) = (-1)^m \sgn(\tilde{z}_1\cdots \tilde{z}_m) \sgn(m_{\widetilde{P}}) = (-1)^{2m} \sgn(m_{\widetilde{P}}) = \sgn(m_{\widetilde{P}}),$$
  where we used \eqref{eq:zeros_poles_signs}. Similarly, we have that $\sgn(m_Q) = \sgn(m_{\widetilde{Q}})$ so that \eqref{eq:signs_mP_mQ} implies that
  \begin{equation}\label{eq:signs_mP_mQ_2}
    \sgn(m_P) = \sgn(m_Q).
  \end{equation}
  Furthermore, \eqref{eq:zeros_poles_signs} implies
  \begin{equation}\label{eq:zeros_poles_signs_2}
    0 > p_n > z_n > p_{n-1} > z_{n-1} > \dots > p_1 > z_1 > -\infty.
  \end{equation}
  Therefore, $P_m(\cdot)$ and $Q_m(\cdot)$ have no common roots and do not have multiple roots. In particular, the function $R_m(x) = P_m(x)/Q_m(x)$ can be decomposed in partial fractions as
  \begin{equation}\label{eq:partial_fractions_expr}
    R_m(x) = k + \sum_{i=1}^m \frac{c_i}{x-p_i}.
  \end{equation}
  Observe that we already have by \eqref{eq:zeros_poles_signs_2} that $p_i < 0$ for $i=1,\ldots,m$. It remains to show that $k>0$ and $c_i>0$ for $i=1,\ldots,m$. 
  Let us start by handling $k$. Note that by \eqref{eq:signs_mP_mQ_2} 
  $$k = \lim_{x\to\infty} R_m(x) = \lim_{x\to\infty} \frac{P_m(x)}{Q_m(x)} = \lim_{x\to\infty} \frac{m_P \left(1-\frac{z_1}{x}\right) \cdots \left(1-\frac{z_m}{x}\right)}{m_Q \left(1-\frac{p_1}{x}\right) \cdots \left(1-\frac{p_m}{x}\right)} = \frac{m_P}{m_Q} > 0.$$
  To conclude the proof it remains to show that $c_i>0$ for $i=1,\ldots,m$. To this end, observe that $p_1, \dots, p_m$ are simple poles of $R_m(x)$ and that $c_i$ is the residual of $R_m$ associated to $p_i$. Indeed, by \eqref{eq:partial_fractions_expr}, we have that
  $$\lim_{x\to p_i} (x-p_i)R_m(x) = \lim_{x\to p_i} \left[c_i + k(x-p_i) + (x-p_i)\sum_{j\neq i} \frac{c_j}{x-p_j} \right] = c_i.$$
  Furthermore, recall that $Q_m(p_i) = 0$ for $i=1,\ldots,m$, so that
  \begin{align*}
  c_i &= \lim_{x\to p_i} (x-p_i) \frac{P_m(x)}{Q_m(x)} = P_m(p_i) \lim_{x\to p_i} \frac{x-p_i}{Q_m(x)} = P_m(p_i) \lim_{x\to p_i} \frac{x-p_i}{Q_m(x)-Q_m(p_i)}\\
  &= P_m(p_i) \lim_{x\to p_i} \frac{1}{\frac{Q_m(x) - Q_m(p_i)}{x-p_i}} = \frac{P_m(p_i)}{Q_m'(p_i)}.
  \end{align*}
  So $c_i = P_m(p_i)/Q_m'(p_i) > 0$ for $i=1,\ldots,m$. Let us first study the sign of $P_m(p_i)$. By \eqref{eq:zeros_poles_signs_2}, we have $m-i$ values in $\{z_i\}_{i=1}^m$ that are larger than $p_i$, so that
  $$
  \sgn(P_m(p_i)) = \sgn\Bigl(m_P \prod_{j=1}^m (p_i-z_j)\Bigr) = \sgn(m_P) (-1)^{m-i}.
  $$
  Let us now study the sign of $Q_m'(p_i)$. We have that
  $$
  Q_m'(x) = m_Q \sum_{j=1}^m \prod_{k\neq j} (x-p_k) \quad \text{so that}\quad Q_m'(p_i) = m_Q \prod_{j\neq i} (p_i - p_j).$$
  By \eqref{eq:zeros_poles_signs_2} we have $m-i$ values in $\{p_j\}_{j=1}^m$ that are strictly larger than $p_i$, so that
  $$
  \sgn(Q_m'(p_i)) = \sgn\Bigl(m_Q \prod_{j\neq i} (p_i - p_j)\Bigr) = \sgn(m_Q) (-1)^{m-i}.$$
  Therefore, by \eqref{eq:signs_mP_mQ_2}, we have that 
  $$\sgn(c_i) = \frac{\sgn(P_m(p_i))}{\sgn(Q_m'(p_i))} = \frac{\sgn(m_P) (-1)^{m-i}}{\sgn(m_Q)(-1)^{m-i}} = 1,$$
  for $i=1,\ldots,m$. This shows that $c_i > 0$ for $i=1,\ldots,m$ and concludes the proof.
  }
\end{proof}

\begin{proof}[of Proposition \ref{cov_prop}]
Note that
\begin{align*}
 &\frac{1}{(\kappa^2p)^{k}(\kappa^2(1-p)+w^2)}-\sum_{j=1}^{k}\frac{1}{(\kappa^2p)^{k+1-j}(\kappa^2+w^2)^j}\\
 &\quad = \frac{1}{(\kappa^2p)^{k}(\kappa^2(1-p)+w^2)}-\frac{1}{(\kappa^2p)^{k+1}}\sum_{j=1}^{k}\left(\frac{\kappa^{2} p}{\kappa^2+w^2}\right)^j
 =\frac{1}{(\kappa^2+w^2)^{k}(\kappa^2(1-p)+w^2)},
\end{align*}
where we used the closed-form expression of the geometric sum to arrive at the final identity.
 Based on this, we can rewrite the spectral density \eqref{partial_frac} as
\begin{equation*}
\begin{aligned}
f_{m,\alpha}(w)= &A\sigma^2\kappa^{-2\alpha}\bigg[\frac{k \kappa^{2\lfloor \alpha \rfloor}}{(\kappa^2+w^2)^{\lfloor \alpha \rfloor}}\\
 &+\sum_{i=1}^{m}c_i\kappa^{2\lfloor \alpha \rfloor +2}
 \bigg(\sum_{j=1}^{\lfloor \alpha \rfloor}\frac{1}{(\kappa^2p)^{\lfloor \alpha \rfloor}(\kappa^2(1-p_i)+w^2)}
-\frac{1}{(\kappa^2p)^{\lfloor \alpha \rfloor+1-j}(\kappa^2+w^2)^{j}}\bigg)\bigg].
\end{aligned}
\end{equation*}

Finally, the explicit expression for the covariance function \eqref{covrational} thus directly follows by taking the inverse Fourier transform of spectral density $f_{m,\alpha}$ and using the linearity of the inverse Fourier transform. 
\end{proof}

\begin{proof}[of Proposition \ref{ra_prec}] 
We begin by recalling that if $v(\cdot)$ is a stationary Gaussian process on $\mathbb{R}$ with spectral density $h$, then by \cite[Theorem 10.1]{pitt1971markov}, the process $v$ will be a Markov process of order $p$, $p\in \mathbb{N}$, on $\mathbb{R}$, if, and only if, the spectral density of $v$, namely $h$, is a reciprocal of a polynomial of degree $2p$. 

The spectral density of the process $u_0$ is $f_{m,0,\alpha}$ and the spectral density of $u_i$, where $i=1,\ldots,m$, is $f_{m,i,\alpha}$, where $f_{m,0,\alpha}$ and $f_{m,i,\alpha}$ were defined in \eqref{partial_frac}. Moreover, observe that $f_{m,0,\alpha}$ is a reciprocal of a polynomial with degree $2\lfloor\alpha\rfloor$, and $f_{m,i,\alpha}$, $i=1,\ldots, m$, are reciprocals of polynomials of degree $2(\lfloor\alpha\rfloor+1)$, thus $u_0$ is a markov process of order $\lfloor\alpha\rfloor$ and $u_j, j=1,\ldots,m$ are Markov processes of order $\lfloor\alpha\rfloor+1$. 
Now, it is well-known, see \cite{whittle63} and \cite[Theorem 4.7 and 4.8]{lindgren2012stationary}, that $u_0$ also solves
\[
\begin{cases}
k(\kappa+\partial_x)(\kappa^2-\Delta)^{\frac{\lfloor \alpha \rfloor-1}{2}} u_0 = \mathcal{W}&\text{when}~\lfloor \alpha \rfloor~\text{is odd},\\
k(\kappa^2-\Delta)^{\frac{\lfloor \alpha \rfloor}{2}}u_0 =\mathcal{W}&\text{when}~\lfloor \alpha \rfloor~\text{is even},
\end{cases}
\]
in the sense that the solution to the above equation has the same covariance function as $u_0$. Similarly, we also have that $u_i$ solves
\[
\begin{cases}
c_i(\kappa\sqrt{1-p_i}+\partial_x)(\kappa+\partial_x)(\kappa^2-\Delta)^{\frac{\lfloor \alpha \rfloor-1}{2}}u_i=\mathcal{W}, &\text{when}~\lfloor \alpha \rfloor~\text{is odd},\\
c_i(\kappa\sqrt{1-p_i}+\partial_x)(\kappa^2-\Delta)^{\nicefrac{\lfloor \alpha \rfloor}{2}}u_i=\mathcal{W}, &\text{when}~\lfloor \alpha \rfloor~\text{is even},
\end{cases}
\]
on interval $I$. Now, since each process is stationary, we start with the stationary distribution, i.e., $\mv{u}_0=[u_0(0), u'_0(0),\dotsc, u_0^{(\lfloor \alpha \rfloor)}(0)]^{\top}$ and $\mv{u}_i=[u_i(0), u_i'(0),\dotsc, u_i^{(\lceil \alpha \rceil)}(0)]^{\top}$, $i>0$, follow the stationary distribution, and thus, the restrictions of these processes from $\mathbb{R}$ to interval $I$ will also be Markov. Therefore, the tridiagonal structure of the precision matrix follows from the fact that if we consider the open set $O_j=(j-1,j+1)$, $j\geq 2$, $j\in \mathbb{N}$, contained in interval $I$, then, in view of the Markov property we have that for each $i=0,\dotsc,m$, $\mv{u}_{i,j} \perp \mv{u}_{i,k} | \mv{u}_{i,j-1}, \mv{u}_{i,j+1}$ for $k\in I\backslash \{j-1,j,j+1\}$, where $\mv{u}_{0,j}=(u_0(w_j),\dotsc, u_0^{(\lfloor \alpha \rfloor)}(w_j))$ and $\mv{u}_{i,j}=(u_i(w_j),\dotsc, u_i^{(\lceil \alpha \rceil)}(w_j))$, $i=1,2,\dotsc,m$, $w_j\in O_j$. Finally, by using this conditional independences and standard techniques (see, for example, the computations of \cite{rue2005gaussian} for conditional autoregressive models), we obtain the expressions for the local precision matrices $\mv{Q}_{i,j}$, $i,j=1,\ldots,n$.
\end{proof}

\begin{proof}[of Proposition~\ref{prop:chol}.]
First, recall that 
$
N = n \left( m \lceil \alpha \rceil + \max(\lfloor \alpha \rfloor, 1) \right).
$
According to Algorithm 2.9 in \cite{rue2005gaussian}, the computational cost of factorizing an $n \times n$ band matrix with bandwidth $p$ is $n(p^2 + 3p)$, and solving a linear system via back-substitution requires $2np$ floating-point operations. 
Since $\mv{Q}$ is an $N \times N$ matrix with bandwidth $2\lfloor \alpha \rfloor + 1$, the total number of floating-point operations required for the Cholesky factorization is
$
n \left( m \lceil \alpha \rceil + \max(\lfloor \alpha \rfloor, 1) \right) 
\left( (2\lfloor \alpha \rfloor + 1)^2 + 3(2\lfloor \alpha \rfloor + 1) \right).
$
Similarly, the floating-point operations required for solving the linear system are 
$
2n \left( m \lceil \alpha \rceil + \max(\lfloor \alpha \rfloor, 1) \right) 
(2\lfloor \alpha \rfloor + 1).
$
\end{proof}

\begin{proof}[of Proposition~\ref{prop:chol_post}]
By reordering $\bar{\mv{U}}$ by location, i.e. 
$$
\bar{\mv{U}} = [\mv{u}_0(t_1), \mv{u}_1(t_1), \ldots, \mv{u}_m(t_1), \mv{u}_0(t_2), \mv{u}_1(t_2), \ldots]^{\top},
$$
we have that $\mv{Q}_{\bar{\mv{U}}|\mv{y}}$ is an $N\times N$ matrix, with bandwidth $\lceil\alpha\rceil(m+1)$. Following the same strategy as in the proof of Proposition~\ref{prop:chol} gives that the Cholesky factor can be computed in 
$$
n(m\lceil\alpha\rceil + \max(\lfloor\alpha\rfloor,1)
)(\lceil\alpha\rceil^2(m+1)^2 + 3\lceil\alpha\rceil(m+1))
$$
floating point operations, and $2n(m\lceil\alpha\rceil + 1)\lceil\alpha\rceil(m+1)$ floating point operations are needed for solving a linear system through back-substitution.
Computing the reordering of $\bar{\mv{U}}$ and reordering the results back can clearly be done in $\mathcal{O}(N)$ cost as the reordering is explicitly known. 
\end{proof}
\bibliography{graph}

\end{document}